\documentclass{amsart}
\usepackage{amsfonts,amssymb,mathtools}
\usepackage[usenames,dvipsnames]{xcolor}

\usepackage{hyperref}

\hypersetup{
    bookmarks=true,         	
    unicode=false,          		
    pdftoolbar=true,        		
    pdfmenubar=true,        	
    pdffitwindow=false,     		
    pdfstartview={FitH},    		
    linktocpage=true,			
    pdfnewwindow=true,      	
    colorlinks=true,       			
    linkcolor=red,          		
    citecolor=PineGreen,    	
    filecolor=magenta,      		
    urlcolor=cyan           		
}

\usepackage{tikz}
\usepackage{float}

\usetikzlibrary{arrows, positioning, calc, intersections}
\usetikzlibrary{decorations.pathreplacing, decorations.markings}

\theoremstyle{plain}
\newtheorem{theorem}{Theorem}[section]

\newtheorem{lemma}[theorem]{Lemma}
\newtheorem{proposition}[theorem]{Proposition}

\theoremstyle{definition}
\newtheorem{definition}[theorem]{Definition}
\newtheorem{problem}[theorem]{Problem}

\theoremstyle{remark}
\newtheorem{remark}[theorem]{Remark}

\numberwithin{figure}{section}
\numberwithin{equation}{section}

\allowdisplaybreaks

\DeclareMathOperator{\ad}{ad}

\DeclareMathOperator{\imag}{Im}

\DeclareMathOperator{\Res}{Res}

\DeclareFontFamily{U}{mathx}{\hyphenchar\font45}
\DeclareFontShape{U}{mathx}{m}{n}{
      <5> <6> <7> <8> <9> <10>
      <10.95> <12> <14.4> <17.28> <20.74> <24.88>
      mathx10
      }{}
\DeclareSymbolFont{mathx}{U}{mathx}{m}{n}
\DeclareFontSubstitution{U}{mathx}{m}{n}
\DeclareMathAccent{\widecheck}{0}{mathx}{"71}
\DeclareMathAccent{\wideparen}{0}{mathx}{"75}

\begin{document}

%
%

\newcommand{\rarr}{\rightarrow}

\newcommand{\bbC}{\mathbb{C}}
\newcommand{\bbR}{\mathbb{R}}
\newcommand{\bfC}{\mathbf{C}}
\newcommand{\calB}{\mathcal{B}}
\newcommand{\calC}{\mathcal{C}}
\newcommand{\calF}{\mathcal{F}}
\newcommand{\calI}{\mathcal{I}}
\newcommand{\calK}{\mathcal{K}}
\newcommand{\calL}{\mathcal{L}}
\newcommand{\calN}{\mathcal{N}}
\newcommand{\calR}{\mathcal{R}}
\newcommand{\calS}{\mathcal{S}}
\newcommand{\calZ}{\mathcal{Z}}

\newcommand{\wbar}{\overline{w}}
\newcommand{\zbar}{\overline{z}}
\newcommand{\etabar}{\overline{\eta}}

\newcommand{\btheta}{\overline{\vartheta}}

\newcommand{\ba}{\breve{a}}
\newcommand{\bb}{\breve{b}}
\newcommand{\bfA}{\mathbf{A}}
\newcommand{\bfM}{\mathbf{M}}
\newcommand{\balpha}{\breve{\alpha}}
\newcommand{\bbeta}{\breve{\beta}}
\newcommand{\brpsi}{\breve{\psi}}
\newcommand{\brho}{\breve{\rho}}
\newcommand{\brm}{\breve{m}}
\newcommand{\bgamma}{\breve{\gamma}}

\newcommand{\dotarg}{\, \cdot \, }

\newcommand{\eps}{\varepsilon}
\newcommand{\lam}{\lambda}
\newcommand{\calJ}{\mathcal{J}}
\newcommand{\calV}{\mathcal{V}}
\newcommand{\calM}{\mathcal{M}}
\newcommand{\calW}{\mathcal{W}}
\newcommand{\tildJ}{\tilde{J}}
\newcommand{\tildV}{\tilde{V}}
\newcommand{\qbar}{\overline{q}}

\newcommand{\alphabar}{\overline{\alpha}}
\newcommand{\betabar}{\overline{\beta}}
\newcommand{\lambar}{\overline{\lambda}}
\newcommand{\zetabar}{\overline{\zeta}}
\newcommand{\rhobar}{\overline{\rho}}

\newcommand{\bfN}{\mathbf{N}}

\newcommand{\cals}{\mathcal{s}}

\newcommand{\darr}{\downarrow}
\newcommand{\lambdabar}{\overline{\lambda}}
\newcommand{\mubar}{\overline{\mu}}
\newcommand{\dee}{\partial}
\newcommand{\dbar}{\overline{\partial}}
\newcommand{\ubar}{\overline{u}}
\newcommand{\vbar}{\overline{v}}

\newcommand{\bfm}{\mathbf{m}}
\newcommand{\bfs}{\mathbf{s}}
\newcommand{\dint}{\displaystyle{\int}}

\newcommand{\Twomat}[4]
{
	\left(
		\begin{array}{ccc}
			#1	&&	#2	\\
			\\
			#3	 &&	#4
		\end{array}
	\right)
}

\newcommand{\twomat}[4]
{
	\left(
		\begin{array}{cc}
			#1	&	#2	\\
			#3	 &	#4
		\end{array}
	\right)
}

\newcommand{\threemat}[9]
{
	\left(
		\begin{array}{ccc}
			#1	 &	#2 & #3	\\
			#4	 & #5 & #6 \\
			#7	 &	#8 & #9
		\end{array}
	\right)
}

\newcommand{\twodet}[4]
{
	\left|
		\begin{array}{cc}
			#1	&	#2	\\
			#3	 &	#4
		\end{array}
	\right|
}

\newcommand{\Twodet}[4]
{
	\left|
		\begin{array}{ccc}
			#1	&&	#2	\\
			\\
			#3	 &&	#4
		\end{array}
	\right|
}

\newcommand{\twovec}[2]
{
	\left(
		\begin{array}{c}
			#1		\\
			#2
		\end{array}
	\right)
}

\newcommand{\Twovec}[2]
{
	\left(
		\begin{array}{c}
			#1		\\
			\\
			#2
		\end{array}
	\right)
}

\newcommand{\norm}[2]{\left\| #1 \right\|_{#2}}
\newcommand{\bigO}[1]{\mathcal{O}\left( #1 \right)}

\title[$L^2$-Sobolev space Bijectivity]{$L^2$-Sobolev space bijectivity of the inverse scattering of a  $3 \times 3$  AKNS system}

\author{Jiaqi Liu}
\address[Liu]{Department of Mathematics, University of Toronto, Toronto, Ontario M5S 2E4, Canada}
\date{\today}
\begin{abstract}

We study the $L^2$-Sobolev space bijectivity of the direct and inverse scattering of the $3\times 3$ AKNS system associated to the Manakov system and Sasa-Satsuma equation. We establish the bijectivity on the weighted Sobolev space $H^{i,1}$ for $i=1,2$. 
\end{abstract}
\maketitle
\tableofcontents

\tikzset{->-/.style={decoration={
  markings,
  mark=at position .55 with {\arrow{triangle 45}} },postaction={decorate}}
  }

\section{Introduction}
We study the following $3\times 3$ AKNS system:
\begin{align}
\label{lax-x}
\psi_x&=L \psi=i\lambda \sigma \psi + U\psi
\end{align}
where
$$ \sigma= \threemat{-1}{0}{0}{0}{1}{0}{0}{0}{1} 
$$
and 
$$U(x)=\threemat{0 }{u}{v}{-\eps u^*}{0}{0}{-\eps v^*}{0}{0}.$$
Throughout this paper the $*$ sign refers to complex conjugation. This linear spectral problem is associated to the Manakov system \cite{Man73}, which is also known as the vector nonlinear Schr\"odinger equation:
\begin{subequations}
\label{MA}
\begin{equation}
iu_t+\dfrac{1}{2}u_{xx}+\eps(|u|^2+|v|^2)u =0
\end{equation}
\begin{equation}
iv_t+\dfrac{1}{2}v_{xx}+\eps(|u|^2+|v|^2)v =0
\end{equation}
\end{subequations}
Here $\eps=\pm 1$ where the $+(-)$ sign denotes the focusing (defocusing) case respectively. Like the scalar NLS equaition, the Manakov system is a universal model for the evolution of weakly nonlinear dispersive wave trains. It appears in many physical contexts, such as deep water waves, nonlinear optics, acoustics, Bose-Einstein condensation. In addition, a reduced version of \eqref{lax-x} also occurs in the inverse scattering transform for the Sasa-Satsuma equation \cite{SS91}:
\begin{equation}
\label{sasa}
u_t-u_{xxx}-6|u|^2u_x-3u(|u|^2)_x=0.
\end{equation}
This equation has particular use in describing the propagation of short pulses
in optical fibers \cite{SS91}.

It has been shown in \cite{PZK15} and \cite{L97} respectively that the generalized Manakov system and the Sasa-Satsuma equation are globally well-posed in $H^{1}(\bbR)$ and $H^2(\bbR)$ respectively. And the soliton-free long time asymptotics of \eqref{MA} and \eqref{sasa} with Schwartz initial conditions have been studied by \cite{GL18}, \cite{GL18-2} respectively. 

\eqref{lax-x} is a $3\times 3$ analogue to the $2\times 2$ AKNS system associated to the scalar nonlinear Schr\"odinger equation in the sense that we are able to formulate the transition matrix and transmission and reflection coefficients by solving the direct scattering problem. This may not be true for general $n\times n$ systems. In \cite{Zhou98} the author proved $L^2$-Sobolev space bijectivity for the inverse scattering transform of this $2\times 2$ AKNS system. More explicitly, for $i, j \geq 1$ given any 
$$q_0\in H^{i,j}(\bbR)=\lbrace f(x): f^{(i)}(x), x^j f(x) \in L^2(\bbR)\rbrace$$
the potential $q$ reconstructed from solving the direct and inverse scattering problem of the system
$$\psi_x=iz\twomat{\frac{1}{2}}{0}{0}{\frac{1}{2}}\psi+\twomat{0}{q_0(x)}{-\eps q_0^*(x)}{0}\psi$$
also belongs to $H^{i,j}(\bbR)$.

The goal of this paper is to establish the bijectivity of the inverse scattering transform for the $3\times 3$ system \eqref{lax-x} between certain $L^2$-Sobolev spaces. This is the first step towards establishing the bijectivity of the scattering-inverse scattering transform for the general $n\times n$ system. The result of this paper will also provide building blocks for rigorous study of the long time behavior, $N$-soliton stability in particular, of the corresponding integrable PDEs.  

The choice of different spaces is determined by the time flow of the specific integrable system. For example, the $t$ problem of the Lax pair associated to the Manakov system is 
\begin{align}
\label{lax-t}  
\psi_t &= i \lambda L \psi+ \dfrac{1}{2}\sigma(U_x+U^2)\psi          
\end{align}
thus the scattering data has time evolution
$$t\mapsto e^{t\lambda^2 \sigma} V(\lambda) e^{-t\lambda^2 \sigma}. $$
In order for the scattering data to persist in the same Sobolev space, we may choose $H^{1,1}(\bbR)$ as the space for the initial data $U_0$. On the other hand, the time evolution of the scattering data associated to the flow \eqref{sasa} is
$$t\mapsto e^{t\lambda^3 \sigma} V(\lambda) e^{-t\lambda^3 \sigma}$$
so we may choose $H^{2,1}(\bbR)$ as the space for the initial data when solving the Cauchy problem of \eqref{sasa}. Since the time evolution of the scattering data is explicit, for simplicity, throughout this paper we will suppress the $t$ variable.

One of the key issues of inverse scattering transform is to properly treat eigenvalues and spectral singularities introduced by the poles of the transmission coefficient (Remark \ref{cases}). We follow the structure of \cite{Zhou98} and separate the issue into three different cases. Each case has its own mathematical implications:
\begin{enumerate}
\item The first one holds for the defocusing case or for $\left\Vert U(x)  \right\Vert_{L^1}$ small in the focusing case. The solution to the Cauchy problem will only consist of a pure radiation background. 
\item The second one is the generic case which we will use to study the long time behavior of the solution with the presence of $N$ solitons. Compared with \cite[Appendix A]{K14}, we  prove the solvability of the Riemann-Hilbert problem in the presence of finitely many discrete spectra. 
\item The third one is the most general case for establishing the $L^2$-Sobolev space bijectivity. Here we give a self-contained and constructive proof of the solvability of Riemann-Hilbert problem defined on contour with self-intersections. 
\end{enumerate}
We state the main result here:
\begin{theorem}
\label{Theorem-main-1}
For $i=1,2$, given initial data $u_0, v_0\in H^{i,1}(\bbR)$, the reconstructed potential $u, v$ obtained through solving the direct and inverse scattering problem of \eqref{lax-x},  are also in 
$H^{i,1}(\bbR)$.
\end{theorem}
\begin{proof}
For $i=1$, the results from Propositions \ref{D1}, \ref{prop-decay}, \ref{smooth} will establish the bijectivity of the first case above. The results from Propositions \ref{D2}, \ref{smooth-1} and  the results from Propositions \ref{scattering data}, \ref{prop-decay-1}, \ref{prop-smooth-1} will  establish the bijectivity of the second and third case respectively. For $i=2$, we will add Remarks \ref{lam^2}, \ref{smooth^2}.
\end{proof}

\begin{remark}
The linear spectral problem associated to \eqref{sasa} is often written as
\begin{equation}
\label{lax-sasa}
\psi_x=i\lambda \threemat{-1}{0}{0}{0}{-1}{0}{0}{0}{1} \psi + \threemat{0 }{0}{u}{0}{0}{u^*}{-u^*}{-u}{0}\psi
\end{equation}
One can simply rotate by 180 degrees and make the change of variable $\lambda\to -\lambda$ to convert \eqref{lax-sasa} to \eqref{lax-x}. 
\end{remark}
\begin{remark}
The  direct and inverse scattering maps in Theorem \ref{Theorem-main-1} are Lipschitz continuous between the corresponding $L^2$-Sobolev spaces. To establish the Lipschitz continuity,  we will use Proposition \ref{lemma-resol} which consists of a compact embedding argument similar to that of \cite[Section 4]{JLPS18} and the second resolvent identity. We leave out the details for simplicity.
\end{remark}
\section{Direct and Inverse Scattering Formalism}
 We begin with some matrix operation notations. Let $B$ be any $3\times 3$ matrix and $\ad\sigma $ the commutator, then
$$\ad \sigma (B)= \sigma B-B\sigma=\threemat{0}{-2B_{12}}{-2B_{13}}{2B_{21}}{0}{0}{2B_{31}}{0}{0}.$$
Also
$$B_x:=e^{i\lambda x\ad\sigma}B =\threemat{B_{11}}{  e^{-2i\lambda x} B_{12}}{  e^{-2i\lambda x} B_{13} } {e^{2i\lambda x} B_{21}}{ B_{22}}{B_{23}} {e^{2i\lambda x} B_{31}}{B_{32}}{B_{33}}$$

In  \cite[Chapter 2]{K14} the author developed the direct and inverse scattering transform for the Manakov system with zero boundary conditions. Analyticity properties of the Jost solutions and symmetry reductions are obtained under the condition $U\in L^1$. The author also  set up the Riemann-Hilbert problem for the inverse scattering problem. We list the main results of  \cite[Chapter 2]{K14} in this section and will mostly follow the its notation convention when studying the $L^2$-Sobolev space bijectivity. 

Note that the Jost solutions to Equation \eqref{lax-x} have two normalizations:
$$
\psi^\pm\rightarrow e^{i\lambda x \sigma}, \quad x\rightarrow \pm\infty
$$
Making the normalization
$$\psi^\pm(x,\lambda)=m^\pm(x,\lambda) e^{ix\lambda\sigma}$$
where
$$m^\pm\rightarrow I,\quad x\rightarrow\pm\infty $$
\eqref{lax-x} then becomes 
\begin{equation}
\label{m}
m_x^\pm= i\lambda  \ad \sigma m^\pm+ U(x)m^\pm.
\end{equation}
and  \eqref{m} is equivalent to the following integral equations:
\begin{subequations}
\begin{equation}
\label{m+}
m^+(x,\lambda)=I+\int_{+\infty}^x e^{i\lambda (x-y)\ad\sigma}U(y)m^+(y,\lambda)dy
\end{equation}
\begin{equation}
\label{m-}
m^-(x,\lambda)=I+\int_{-\infty}^x e^{i\lambda (x-y)\ad\sigma}U(y)m^-(y,\lambda)dy
\end{equation}
\end{subequations}
Through the uniqueness theory of ordinary differential equation, for $\lambda\in \bbR$, the two Jost solutions are related by a transition matrix
\begin{align}
\label{S}
\psi^-(x,\lambda)&=\psi^+(x,\lambda)S(\lambda)\\
                           &=\psi^+(x,\lambda)\threemat{s_{11}}{s_{12}}{s_{13}}{s_{21}}{s_{22}}{s_{23}}{s_{31}}{s_{32}}{s_{33}}
\end{align}

\begin{align}
\label{S-m}
m^-(x,\lambda)&=m^+(x,\lambda)e^{ix\lambda\ad\sigma}S(\lambda)\\
\nonumber
                           &=m^+(x,\lambda)\threemat{s_{11}}{ e^{-2i\lambda x} s_{12}}{e^{-2i\lambda x} s_{13}}{e^{2i\lambda x} s_{21}}{s_{22}}{s_{23}}{e^{2i\lambda x} s_{31}}{s_{32}}{s_{33}}
\end{align}
Jacobi's formula implies that $$\text{det}\psi^\pm(x,\lambda)=e^{i\lambda x}$$
from which it is easily seen that $\text{det}S(\lambda)=1$.  We also define
\begin{equation}
\label{T}
T(\lambda)=S^{-1}(\lambda).
\end{equation}
By standard Volterra theory (see \cite[Appendix A.1]{K14}), we have the following analyticity properties:
\begin{proposition}
\label{analyticity-m}
For $U(x) \in L^1(\bbR)$, the following columns of $m^\pm(x,\lambda)$ can be analytically extended into the corresponding half plane:
\begin{subequations}
\begin{equation}
m^-_1(x,\lambda), m^+_2(x,\lambda), m^+_3(x,\lambda): \quad \imag \lambda>0,
\end{equation}
\begin{equation}
m^+_1(x,\lambda), m^-_2(x,\lambda), m^-_3(x,\lambda): \quad \imag \lambda<0.
\end{equation}
\end{subequations}
\end{proposition}
As a consequence, we have the analytic extension for columns of $\psi^\pm$:
\begin{subequations}
\begin{equation}
\psi^-_1(x,\lambda), \psi^+_2(x,\lambda), \psi^+_3(x,\lambda): \quad \imag \lambda>0,
\end{equation}
\begin{equation}
\psi^+_1(x,\lambda), \psi^-_2(x,\lambda), \psi^-_3(x,\lambda): \quad \imag \lambda<0.
\end{equation}
\end{subequations}
\begin{proposition}
\label{analyticity-S}
For $U(x) \in L^1(\bbR)$, $S(\lambda)$ has the following integral representation:
\begin{equation}
\label{S-int}
S(\lambda)=I+\int_\bbR e^{-i\lambda y\ad\sigma}U(y)m^-(y,\lambda)dy
\end{equation}
and the entries of $S(\lambda)$, $T(\lambda)$ have the following analytic extensions:
\begin{subequations}
\begin{equation}
\label{analyticity-s+}
s_{11}(\lambda), t_{22}(\lambda), t_{23}(\lambda), t_{32}(\lambda), t_{33}(\lambda): \quad \imag \lambda>0,
\end{equation}
\begin{equation}
\label{analyticity-s-}
t_{11}(\lambda), s_{22}(\lambda), s_{23}(\lambda), s_{32}(\lambda), s_{33}(\lambda): \quad \imag \lambda<0,
\end{equation}
\end{subequations}
\end{proposition}

\subsection{Symmetry Reductions}
\begin{proposition}{\cite[Lemma 2.14]{K14}}
If $\psi(x,\lambda)$ is a fundamental solution of the linear spectral problem \eqref{lax-x}, then  $J_\eps (\psi^\dagger(x,\lambda))^{-1}$ is also a fundamental solution where
$$J_\eps=\threemat{1}{0}{0} {0}{\eps}{0} {0}{0}{\eps}$$
and $\dagger$ refers to Hermitian adjoint. 
 Combining with the normalization of $\psi^{\pm}(x,\lambda)$ we have 
 \begin{equation}
 \label{symm-psi}
 \psi^{\pm}(x,\lambda)=J_\eps (\psi^\dagger(x,\lambda))^{-1}J_\eps.
 \end{equation}
\end{proposition}
We can expand \eqref{symm-psi} to obtain the following symmetry reductions:
\begin{subequations}
\begin{equation}
\psi^\pm_1(x,\lambda^*)^*=e^{-i\lambda x}J_\eps  \left[ \psi_2^\pm\times\psi_3^\pm  \right](x,\lambda),  \quad \imag k \gtrless 0,
\end{equation}
\begin{equation}
\psi^\pm_2(x,\lambda)^*=\eps e^{-i\lambda x}J_\eps  \left[ \psi_3^\pm\times\psi_1^\pm  \right](x,\lambda),  \quad k\in\bbR,
\end{equation}
\begin{equation}
\psi^\pm_3(x,\lambda)^*=\eps e^{-i\lambda x}J_\eps  \left[ \psi_1^\pm\times\psi_2^\pm  \right](x,\lambda),  \quad k\in\bbR
\end{equation}
\end{subequations}

Using the scattering relation \eqref{S} it is straightforward to check that for $\lambda\in\bbR$
\begin{equation}
\label{S-T}
S(\lambda)=J_\eps [T(\lambda)]^\dagger J_\eps
\end{equation}
and componentwise for $\lambda\in\bbR$ we have that
\begin{subequations}
\begin{equation}
\label{s-t-1}
s_{11}(\lambda)=t^*_{11}(\lambda), \, s_{12}(\lambda)=\eps t^*_{21}(\lambda), \, s_{13}(\lambda)=\eps t^*_{31}(\lambda),
\end{equation}
\begin{equation}
\label{s-t-2}
s_{21}(\lambda)=\eps t^*_{12}(\lambda), \, s_{22}(\lambda)=t^*_{22}(\lambda), \, s_{23}(\lambda)= t^*_{32}(\lambda),
\end{equation}
\begin{equation}
\label{s-t-3}
s_{31}(\lambda)=\eps t^*_{13}(\lambda), \, s_{32}(\lambda)= t^*_{23}(\lambda), \, s_{33}(\lambda)= t^*_{33}(\lambda).
\end{equation}
\end{subequations}
Using the Schwarz reflection principle and analyticity properties given by  \eqref{analyticity-s+}-\eqref{analyticity-s-}, we can deduce the following symmetry reductions:
$$s_{11}(\lambda)=t^*_{11}(\lambda^*),\,t_{22}(\lambda)=s^*_{22}(\lambda^*),\, t_{23}(\lambda)=s^*_{32}(\lambda^*),\, t_{33}(\lambda)=s^*_{33}(\lambda^*).$$
An important relation can be deduced from symmetries above and the fact that 
$\text{det}S(\lambda)=\text{det}T(\lambda)=1$:
\begin{equation}
\label{detS}
|s_{11}(\lambda)|^2+\eps( |s_{21}(\lambda)|^2+|s_{31}(\lambda)|^2)=1.
\end{equation}

\subsection{Beals-Coifman Solutions}
From Proposition \ref{analyticity-m} we can construct the following matrices with analyticity in $\bbC^\pm$:
\begin{subequations}
\begin{equation}
\label{Psi+}
\Psi^+(x,\lambda)=\left(\psi^-_1(x,\lambda),  \psi^+_2(x,\lambda), \psi^+_3(x,\lambda)    \right)
\end{equation}
\begin{equation}
\label{Psi-}
\Psi^-(x,\lambda)=\left(\psi^+_1(x,\lambda),  \brpsi^-_2(x,\lambda), \brpsi^-_3(x,\lambda)    \right)
\end{equation}
\end{subequations}
where $ \brpsi^-_2$ and $ \brpsi^-_3$ in \eqref{Psi-} are given by
\begin{subequations}
\begin{equation}
\label{brpsi-2}
\brpsi^-_2(x,\lambda)=-\eps e^{i\lambda x} J_\sigma \left(  m^{-*}_1 \times m^{+*}_3 \right)(x,\lambda^*),\quad \lambda\in\bbC^-
\end{equation}
\begin{equation}
\label{brpsi-3}
\brpsi^-_3(x,\lambda)=\eps e^{i\lambda x} J_\sigma \left(  m^{-*}_1 \times m^{+*}_2 \right)(x,\lambda^*) ,\quad \lambda\in\bbC^-.
\end{equation}
\end{subequations}
Using the relation given by \eqref{S} and \eqref{S-m} and the symmetries \eqref{s-t-1}-\eqref{s-t-3} we find 
\begin{subequations}
\begin{equation}
\label{brpsi-2'}
\brpsi^-_2(x,\lambda)= s_{33}(\lambda)\psi^-_2(x,\lambda)-s_{32}(\lambda)\psi^-_3(x,\lambda)
\end{equation}
\begin{equation}
\label{brpsi-3'}
\brpsi^-_3(x,\lambda)=-s_{23}(\lambda)\psi^-_2(x,\lambda)+s_{22}(\lambda)\psi^-_3(x,\lambda).
\end{equation}
\end{subequations}
\begin{remark}
\eqref{brpsi-2}-\eqref{brpsi-3} are useful for establishing the symmetry condition required by the vanishing lemma. We will use \eqref{brpsi-2'}-\eqref{brpsi-3'} to write down the appropriate jump matrix over $\bbR$.
\end{remark}
Straightforward computation gives 
\begin{subequations}
\begin{equation}
\label{det-s11}
\det \Psi^+(x,\lambda)=s_{11}(\lambda)e^{i\lambda x}
\end{equation}
\begin{equation}
\label{det-t11}
\det \Psi^-(x,\lambda)=(t_{11}(\lambda))^2e^{i\lambda x}.
\end{equation}
\end{subequations}
\begin{definition}
The discrete eigenvalues consist of the zeros $\lbrace \lambda_i  \rbrace_{i=1}^n$ of $s_{11}(\lambda)$ in $\bbC^+$ or by symmetry relation \eqref{s-t-1} the zeros $\lbrace \lambda_i^*  \rbrace_{i=1}^n$ of $t_{11}(\lambda)$ in $\bbC^-$. The spectral singularities are the zeros of $s_{11}(\lambda)$ for $\lambda\in \bbR$. 
\end{definition}
\begin{remark}
For $\eps=-1$ , from \eqref{detS}  it is easily seen that $|s_{11}|\geq 1$ for all $\lambda\in \bbR$. Also notice that the linear operator 
$-i\sigma{d}/{dx} +i\sigma U(x)$
is self-adjoint so that there is no real $L^2$-eigenvalues for the spectral problem
$$-i\sigma\dfrac{d}{dx} \psi+i\sigma U(x)\psi=\lambda \psi.$$
These together imply that $|s_{11}(\lambda)|>0$ in $\bbC^+$. Thus there are no discrete eigenvalues and spectral singularities for the defocusing case. The same conclusion does not hold for the focusing case thus both discrete eigenvalues and spectral singularities are allowed. 
\end{remark}
To define the Beals-Coifman solutions, we need the following notations:
\begin{align}
\psi^{++}(x,\lambda)&=\left( \psi^+_2(x,\lambda), \psi^+_3(x,\lambda) \right)\\
m^{++}(x,\lambda)&=e^{-i\lambda x}\psi^{++}(x,\lambda)\\
\nonumber
                             &=\left( m^+_2(x,\lambda), m^+_3(x,\lambda) \right)
\end{align}
\begin{align}
\brpsi^{--}(x,\lambda)&=\left( \brpsi^-_2(x,\lambda), \brpsi^-_3(x,\lambda) \right)\\
\brm^{--}(x,\lambda)&=e^{-i\lambda x}\brpsi^{--}(x,\lambda)
\end{align}
We now define the Beals-Coifman solutions as follows:
\begin{subequations}
\begin{equation}
\label{BC+}
M^+(x,\lambda)=\left( \dfrac{m^-_1(x,\lambda)}{s_{11}(\lambda)}, m^{++}(x,\lambda) \right) \quad \imag\lambda\geq 0
\end{equation}
\begin{equation}
\label{BC-}
M^-(x,\lambda)=\left( {m^+_1(x,\lambda)}, \dfrac{ \brm^{--}(x,\lambda)}{t_{11}(\lambda)} \right), \quad \imag\lambda\leq 0
\end{equation}
\end{subequations}
The Beals-Coifman solution has the following properties:
\begin{enumerate}
\item $\det M^\pm=1$;
\item  $\lim_{x\to +\infty} M^\pm(x,\lambda)=I$.
\end{enumerate}

\begin{remark}
\label{cases}
Since $s_{11}(\lambda)$ is on the denominator, there are three separate cases regarding the zeros of $s_{11}(\lambda)$:
\begin{enumerate}
\item[I] $s_{11}(\lambda)$ has no zero in $\bbC^+\cup \bbR$. 
\item[II] $s_{11}(\lambda)$ has finitely many first order zeros in $\bbC^+$. 
\item[III] $s_{11}(\lambda)$ has (possibly infinitely many) zeros  of arbitrary order in $\bbC^+\cup \bbR$.
\end{enumerate}
\end{remark}
For now we assume $s_{11}(\lambda)$ has no zero in $\bbC^+\cup \bbR$ and find that the Beals-Coifman solution \eqref{BC+}-\eqref{BC-} satisfies the following jump condition on $\bbR$:
\begin{equation}
\label{BC-jump}
M_+(x,\lambda)=M_-(x,\lambda)e^{i\lambda x\ad\sigma}V(\lambda)
\end{equation}
where
\begin{equation}
V(\lambda)=\threemat{1+\eps |\rho_1(\lambda)|^2+\eps|\rho_2(\lambda)|^2}{\eps \rho^*_1(\lambda)}{\eps \rho^*_2(\lambda)}{\rho_1(\lambda)}{1}{0}{\rho_2(\lambda)}{0}{1}
\end{equation}
with 
\begin{equation}
\label{rho1,2}
\rho_1(\lambda):=\dfrac{s_{21}(\lambda)}{s_{11}(\lambda)},\quad \rho_2(\lambda):=\dfrac{s_{31}(\lambda)}{s_{11}(\lambda)}
\end{equation}
and by symmetry
$$\rho_1^*(\lambda)=\dfrac{s_{21}^*(\lambda)}{s_{11}^*(\lambda)}=\dfrac{\eps t_{12}(\lambda)}{t_{11}(\lambda)},\quad \rho_2^*(\lambda)=\dfrac{s_{31}^*(\lambda)}{s_{11}^*(\lambda)}=\dfrac{\eps t_{13}(\lambda)}{t_{11}(\lambda).}$$
\subsection{An Auxiliary Scattering Matrix}
Notice that by Remark \ref{BC-M} the Beals-Coifman solution is normalized at $x=+\infty$. Also we can modify the Beals-Coifman solution  so that it is normalized at $x=-\infty$. To do this, we define the following auxiliary scattering matrix
\begin{equation}
\label{auxiliary}
A(z)=\begin{cases}
\threemat{s_{11}}{0}{0} {0}{\dfrac{t_{33}}{s_{11}}}{-\dfrac{t_{23}}{s_{11}}} {0}{-\dfrac{t_{32}}{s_{11}}}{\dfrac{t_{22}}{s_{11}}} &\quad\imag z\geq 0\\
\\
\threemat{\dfrac{1}{t_{11} } }{0}{0} {0}{s_{22}}{s_{23}} {0}{s_{32}}{s_{33}} &\quad\imag z\leq 0
\end{cases}
\end{equation}
and obtain
\begin{equation}
\label{m-tilde}
\tilde{M}(x,z)=M(x,z)A(z).
\end{equation}
and the jump relation of $\tilde{M}^\pm(x,z)$ is given by
\begin{equation}
\label{BC-jump-til}
\tilde{M}_+(x,\lambda)=\tilde{M}_-(x,\lambda)e^{i\lambda x\ad\sigma}\tilde{V}(\lambda)
\end{equation}
where
\begin{align}
\label{V-tilde}
\tilde{V}(\lambda) &=A^{-1}_-V(\lambda)A_+(\lambda)\\
\nonumber
                             &=\threemat{1}{\eps \tilde{\rho_1}^*}{ \eps\tilde{\rho_2}^*} { \tilde{\rho_1}}{1+ \eps |\tilde{\rho}_1|^2}{ \eps \tilde{\rho_2}^* \tilde{\rho_1} } {\tilde{\rho}_2 }{ \eps \tilde{\rho_1}^*\tilde{\rho_2} }{1+ \eps |\tilde{\rho}_2|^2}
\end{align}
where 
$$\tilde{\rho_1}=-\dfrac{t_{21}}{t_{11}},\quad \tilde{\rho_2}=-\dfrac{t_{31}}{t_{11}}.$$

\section{Case I: Pure Radiation}
In the following three sections, we  study the three cases listed on Remark \ref{cases} separately. In this section, for case I, we will establish the following Sobolev space bijectivity of   the direct scattering and inverse scattering map:
$$\mathcal{D}: H^{1,1}(\bbR)\ni\lbrace u_0, v_0 \rbrace\longmapsto \lbrace \rho_1(\lambda), \rho_2(\lambda) \rbrace\in H^{1,1}(\bbR)$$
$$\mathcal{I}: H^{1,1}(\bbR)\ni\lbrace\rho_1(\lambda), \rho_2(\lambda) \rbrace\longmapsto \lbrace u, v \rbrace\in H^{1,1}(\bbR).$$
This will provide building blocks for the study of the remaining two cases.
\subsection{Direct Scattering}
\begin{proposition}
\label{D1}
If $u_0, v_0\in H^{1,1}(\bbR)$,  then $\rho_1(\lambda), \rho_2(\lambda)\in H^{1,1}(\bbR)$.
\end{proposition}
\begin{proof}
Given the formulation of $\rho_1, \rho_2$ by \eqref{rho1,2} and the fact that $|s_{11}|>1$, to prove Proposition \ref{D1}, we need to establish the following mapping properties:
\begin{subequations}
\begin{equation}
\label{D:s11}
H^{1,1}(\bbR)\ni\lbrace u_0, v_0 \rbrace\longmapsto s_{11}(\lambda)-1\in H^1(\bbR)
\end{equation}
\begin{equation}
\label{D:s21,31}
H^{1,1}(\bbR)\ni\lbrace u_0, v_0 \rbrace\longmapsto \lbrace s_{21}(\lambda), s_{31}(\lambda) \rbrace\in H^{1,1}(\bbR)
\end{equation}
\end{subequations}
Recall that $S(\lambda)$ is given by the  integral equation \eqref{S-int}. Combining  \eqref{S-int} with \eqref{m-} we obtain:
\begin{align}
\label{S-K}
S(\lambda) &=I+\int_\bbR e^{-i\lambda y\ad\sigma}U(y)\left(m^-(y,\lambda) -I\right)dy+\int_\bbR e^{-i\lambda y\ad\sigma}U(y)dy\\
\nonumber
                  &=I+\int_\bbR e^{-i\lambda y\ad\sigma}U(y)\left(I-K^-_u\right)^{-1}K^-_u Idy+\int_\bbR e^{-i\lambda y\ad\sigma}U(y)dy\\
                  &=I+S_{1}(\lambda)+S_{2}(\lambda)
\end{align}
with the integral operator $K^\pm_u$ given by
\begin{equation}
\label{Ksu}
(K^\pm_u h)(x)=\int_{\pm\infty}^x e^{i\lambda(x-y)\ad \sigma} U(y)h(y)dy
\end{equation}
 By standard Fourier theory, it is easily seen that $S_2(\lambda)\in H^{1,1}(\bbR)$ given $U(x)\in H^{1,1}(\bbR)$. For $S_2(\lambda)$, we first note that 
 \begin{align}
 \label{test-phi}
 \left\Vert (K^\pm_u I)(x, \cdot) \right\Vert_{L^2_\lambda}&=\sup_{\substack{\phi\in C_0^\infty\\ \left\Vert \phi \right\Vert_{L^2}=1 }} \left\vert   \int_\bbR \phi(\lambda) (K^\pm_u I)(x,\lambda) d\lambda \right\vert\\
 \nonumber
                                                                                        &\lesssim \sup_{\substack{\phi\in C_0^\infty\\ \left\Vert \phi \right\Vert_{L^2}=1 }} \pm\int_x^{\pm \infty} \left( | \hat{\phi}(x-y)|+|\hat{\phi}(y-x)| \right)\left\vert U(y)\right\vert dy\\
  \nonumber                                                                                      
                                                                                        &\leq \left\Vert U\right\Vert_{L^2}
 \end{align}
Using Minkowski inequality, for \eqref{Ksu} we have 
$$ \norm{(K^\pm_u h)(x, \cdot)}{L^2_\lambda} \leq \left(   \pm\int_{\pm\infty}^x |U(y)|^2 dy \right)\sup_{\pm y\geq\pm x}\norm{h(y, \cdot)}{L^2_\lambda}.$$
We combine this with standard Volterra theory to obtain
\begin{equation}
\label{resolvent-K}
\norm{m^\pm-I}{L^\infty L^2_\lambda}=\norm{\left(I-K^-_u\right)^{-1}K^-_u I}{L^\infty_x L^2_\lambda}\leq e^{\norm{U}{L^1}}\norm{U}{L^2}.
\end{equation}
Another application of the Minkowski inequality implies $S_2(\lambda)\in L^2(\bbR)$.

We now show that $\lambda s_{21}(\lambda), \lambda s_{31}(\lambda)\in L^2(\bbR)$. We perform integration by parts  on $K_u^- I$:
\begin{align}
\label{KI}
K_u^- I &=\int_{-\infty}^x e^{i\lambda(x-y)\ad \sigma} U(y)dy\\
\nonumber
            &=-\dfrac{1}{\lambda}\int_{-\infty}^x (e^{i\lambda(x-y)\ad \sigma})' (i\ad\sigma)^{-1}U(y)dy\\
  \nonumber          
            &=-\dfrac{1}{\lambda}(i\ad\sigma)^{-1}U(x)+\dfrac{1}{\lambda}(i\ad\sigma)^{-1}\int_{-\infty}^x e^{i\lambda(x-y)\ad \sigma} U'(y)dy\\
   \nonumber
            &=h_1(x,\lambda)+h_2(x, \lambda).
\end{align}
We further write
\begin{align*}
(1-K^-_u)^{-1}K^-_u I &=h_1+K_u^- h_1+(1-K^-_u)^{-1}K^{2}_u h_1+ (1-K^-_u)^{-1}h_2\\
                                   &=h_1+g_1+g_2+g_3
\end{align*}
Thus we need to show that  the (2-1) and (3-1) entries of the following matrix
$$\int_\bbR e^{-i\lambda y\ad\sigma}U(y)\lambda (h_1+g_1+g_2+g_3) dy$$
are in $L^2_\lambda(\bbR)$.

It can be easily seen that 
$$\int_\bbR e^{-i\lambda y\ad\sigma}U(y) h_1(y,\lambda)dy$$
is a diagonal matrix thus makes no contribution to $s_{21}$ and $s_{31}$. 

For $\lambda g_3$, as in \eqref{resolvent-K} we have
\begin{equation}
\label{resolvent-k}
\norm{\left(I-K^-_u\right)^{-1}\int_{-\infty}^x (i\ad\sigma)^{-1} e^{i\lambda(x-y)\ad \sigma} U'(y)dy}{L^\infty_x L^2_\lambda}\leq e^{\norm{U}{L^1}}\norm{U'}{L^2}.
\end{equation}
An application of the Minkowski inequality implies $\lambda g_3\in L^2(\bbR)$.

For $g_1$ we notice that by trangularity:
\begin{align*}
\int_\bbR e^{-iy\lambda\ad\sigma}U(y) \lambda g_1(y,\lambda)dy &=\int_\bbR e^{-iy\lambda\ad\sigma}U(y) \left( \int_{-\infty}^y e^{i\lambda(y-z)\ad \sigma} U(z)(i\ad\sigma)^{-1}U(z)dz\right) dy\\
                                                                                                        &=\int_\bbR e^{-iy\lambda\ad\sigma}U(y) \left( \int_{-\infty}^y  U(z)(i\ad\sigma)^{-1}U(z)dz\right) dy.
\end{align*}
Using the same argument as in the proof of  \eqref{test-phi} we conclude that
$$\norm{\int_\bbR e^{-iy\lambda\ad\sigma}U(y) \lambda g_1(y,\lambda)dy}{L^2_\lambda}\lesssim \norm{U}{L^2}^2.$$

Finally for $\lambda g_2$, we first note that by triangularity, 
$$\lambda K^2_u h_1(x,\lambda)=\dfrac{1}{\lambda}\int_{-\infty}^x e^{i\lambda(x-y)\ad \sigma} U(y) \left( \int_{-\infty}^y U(z)(i\ad\sigma)^{-1}U(z)dz\right) dy.$$
Thus
$$\norm{(I-K^-_u)^{-1} \lambda K^2_u h_1} {L^\infty L^2_\lambda}\lesssim e^{\norm{U}{L^1}}\norm{U}{L^2}^2$$
and another application of the Minkowski inequality implies $\lambda g_2\in L^2(\bbR)$.

Using the scattering relation \eqref{S-m} and letting $x=0$, we can write $$S(\lambda)_\lambda=m^+(0,\lambda)^{-1}_\lambda m^-(0,\lambda)+m^+(0,\lambda)^{-1} m^-(0,\lambda)_\lambda.$$
By standard Volterra theory, $\norm{m^\pm(x,\lambda)}{L^\infty_x L^\infty_\lambda}<\infty$. So we only need to show $m^\pm(0,\lambda)_\lambda\in L^2_\lambda(\bbR)$. Here we are going to treat the $\pm$ cases simultaneously. Indeed we only need estimates on $m^+(x,\lambda)$ for $x\geq 0$  and on $m^-(x,\lambda)$ for $x\leq 0$. From \eqref{m+}-\eqref{m-} we compute
\begin{align}
\label{d-m}
\dfrac{\partial m}{\partial\lambda}&=\dfrac{\partial}{\partial\lambda}(K_u I)+\left[  \dfrac{\partial K_u}{\partial\lambda} \right](m-I)+K_u\left( \dfrac{\partial m}{\partial\lambda} \right)\\
                                                    &=H_1(x,\lambda)+H_2(x,\lambda)+K_u\left( \dfrac{\partial m}{\partial\lambda} \right)
\end{align}
From the resolvent bound \eqref{resolvent-K}, we only need to show that $H_1, H_2\in L^\infty_x L^2_\lambda$.
\begin{equation}
\label{K-partial-1}
H_1(x,\lambda)=\int_{\pm\infty}^x e^{i\lambda(x-y)\ad \sigma} (i\ad\sigma)(x-y)U(y)dy
\end{equation}
hence
\begin{align*}
\norm{H_1(x,\lambda)}{L^2_\lambda} &=\sup_{\substack{\phi\in C_0^\infty\\ \left\Vert \phi \right\Vert_{L^2}=1 }} \left\vert   \int_\bbR \phi(\lambda)  \left( \int_{\pm\infty}^x e^{i\lambda(x-y)\ad \sigma} (i\ad\sigma)(x-y)U(y)dy  \right) d\lambda \right\vert\\
                                                           &\lesssim \sup_{\substack{\phi\in C_0^\infty\\ \left\Vert \phi \right\Vert_{L^2}=1 }} \mp\int_x^{\pm \infty} \left( | \hat{\phi}(x-y)|+|\hat{\phi}(y-x)| \right) \left\vert x-y \right\vert   \left\vert U(y)\right\vert dy\\
                                                           &\lesssim \sup_{\substack{\phi\in C_0^\infty\\ \left\Vert \phi \right\Vert_{L^2}=1 }} \mp\int_x^{\pm \infty} \left( | \hat{\phi}(x-y)|+|\hat{\phi}(y-x)| \right) \left\vert y \right\vert   \left\vert U(y)\right\vert dy
\end{align*}
where we used the fact that $|x-y|<|y|$ for both $0<x<y$ and $y<x<0$. Finally an application of Schwarz inequality gives
\begin{equation}
\label{H_1}
\norm{H_1(x,\lambda)}{L^\infty L^2_\lambda}\leq \norm{xU(x)}{L^2}
\end{equation}
For $H_2(x,\lambda)$, note that
\begin{equation}
\label{K-partial-2}
H_2(x,\lambda)=\int_{\pm\infty}^x e^{i\lambda(x-y)\ad \sigma} (i\ad\sigma)(x-y)U(y)(m^\pm(y,\lambda)-I)dy,
\end{equation}
so that, by Minkowski inequality and Schwarz inequality
$$\norm{H_2(x,\lambda)}{L^\infty L^2_\lambda}\leq \left( \mp \int_{\pm\infty}^x|y|^2|U(y)|^2 dy  \right)^{1/2} \norm{m^\pm-I}{L^2L^2_\lambda} .$$
To see $\norm{m^\pm-I}{L^2L^2_\lambda}$ is finite, we write
$$\norm{m^\pm-I}{L^2L^2_\lambda}\leq \norm{K_u I}{L^2L^2_\lambda}+\norm{K_u (m^\pm-I)}{L^2L^2_\lambda}.$$
In \eqref{test-phi} we have shown that 
$$\norm{K_u I(x,\cdot)}{L^2_\lambda}\leq \left( \mp \int_{\pm\infty}^x |U(y)|^2 dy \right)^{1/2} $$
Reversing the order of integration gives
$$
\norm{K_u I}{L^2 L^2_\lambda}\leq \int_{\pm\infty}^0\int_{\pm\infty}^x |U(y)|^2 dydx
=\int_{\pm\infty}^0 |y||U(y)|^2 dxdy<\infty.$$
Finally, using Minkowski’s integral inequality and\eqref{resolvent-K} we can show that
\begin{align*}
\norm{K_u (m^\pm-I)}{L^2_\lambda} &\leq \mp\left( \int_{\pm \infty}^x |U(y)|dy \right)\norm{m^\pm-I}{L^\infty L^2_\lambda}\\
                                                           &\leq  \mp\left( \int_{\pm \infty}^x |U(y)|dy \right)e^{\norm{U}{L^1}}\norm{U}{L^2}
\end{align*}
and an application of the Hardy's inequality gives
\begin{align*}
\norm{K_u (m^\pm-I)}{L^2L^2_\lambda} &\leq \mp \left[\int_{\pm\infty}^0 \left( \int_{\pm \infty}^x |U(y)|^2dy \right)dx \right]^{1/2} e^{\norm{U}{L^1}}\norm{U}{L^2}\\
                                                               &\leq \norm{U(x)}{H^{0,1}}e^{\norm{U}{L^1}}\norm{U}{L^2}
\end{align*}
Now we have established both \eqref{D:s11} and \eqref{D:s21,31} thus complete the proof of Proposition \ref{D1}.
\end{proof}
\begin{remark}
\label{lam^2}
If we further assume that $u_0, v_0\in H^2(\bbR)$, then we are allowed to  write \eqref{KI} as
$$K_u^-I=\dfrac{1}{\lambda^2}\int_{-\infty}^x (e^{i\lambda(x-y)\ad \sigma}) (i\ad\sigma)^{(2)}U(y)dy$$
Integrating by parts twice and the same argument in the proof above will show that $\rho_1, \rho_2\in H^{0,2}(\bbR)$.
\end{remark}

\subsection{Inverse Scattering}
We study the following Riemann-Hilbert problem:
\begin{problem}
\label{RHP-1}
For fixed $x\in\bbR$ and  $\rho_1(\lambda), \rho_2(\lambda)\in H^{1,1}(\bbR)$, find a matrix $M(x,z)$ satisfying the following conditions:
\begin{enumerate}
\item[(i)] (Analyticity) $M(x,z)$ is analytic for $z\in\bbC\setminus\bbR$.
\item[(ii)] (Normalization) $M(x,z)\to I+\mathcal{O}(z^{-1})$ as $z\to \infty$.
\item[(iii)] (Jump relation) For each $\lambda\in \bbR$, $M(x,z)$ has continuous non-tangential boundary value $M_\pm(x,\lambda)$ as $z\to \lambda$ from $\bbC^\pm$ and the following jump relation holds
\begin{equation}
\label{jump}
M^+(x,\lambda)=M^-(x,\lambda)e^{i\lambda x\ad\sigma}V(\lambda)
\end{equation}
where
\begin{equation}
\label{V}
V(\lambda)=\threemat{1+\eps |\rho_1(\lambda)|^2+\eps|\rho_2(\lambda)|^2}{\eps \rho^*_1(\lambda)}{\eps \rho^*_2(\lambda)}{\rho_1(\lambda)}{1}{0}{\rho_2(\lambda)}{0}{1}
\end{equation}
\end{enumerate}
\end{problem}
The jump matrix $V$ admits the following factorization:
\begin{equation}
\label{factorize}
V(\lambda)=V_-(\lambda)^{-1}V_+(\lambda)=\left( I-W_-(\lambda) \right)^{-1}\left(I+W_+(\lambda)\right)
\end{equation}
where 
\begin{equation}
\label{W}
W_-(\lambda)=\threemat{0}{\eps \rho^*_1(\lambda) }{\eps  \rho^*_2(\lambda) } {0}{0}{0} {0}{0}{0},\quad W_+(\lambda)=\threemat{0}{0 }{0} {\rho_1(\lambda)}{0}{0} {\rho_2(\lambda)}{0}{0}
\end{equation}
\begin{proposition}
The Riemann-Hilbert problem \ref{RHP-1} has a unique solution.
\end{proposition}
\begin{proof}
By standard Riemann-Hilbert theory, the existence and uniqueness of the solution to Problem \ref{RHP-1} is determined by the existence and uniqueness of the following singular integral equation:
\begin{align}
\label{SIE-1}
\mu(x,\lambda)&=I+\mathcal{C}_W \mu(x,\lambda)\\
\nonumber
                        &=I+C^+_\bbR(\mu W_{x-})(x,\lambda)+C^-_\bbR(\mu W_{x+})(x,\lambda)
\end{align}
where 
\begin{equation}
\label{mu}
\mu(x,\lambda)=M_+(I+W_{x+})^{-1}=M_-(I-W_{x-})^{-1}
\end{equation}
and $C^\pm$ are the Cauchy projections
$$C^\pm_\bbR f(z)=\lim_{\eps\to 0+}\dfrac{1}{2\pi i}\int_\bbR \dfrac{f(\lambda)}{\lambda-(z\pm i\eps)} d\lambda.$$
It is shown in \cite[Proposition 4.1, 4.2]{Zhou89} that the operator $I-\mathcal{C}_W$ is Fredholm and has Fredholm index zero. It is also easy to check that $V(\lambda)+V(\lambda)^\dagger$ is a positive definite matrix thus $ker(I-\mathcal{C}_W)=0$ by \cite[Proposition 9.3]{Zhou89}. The operator $I-\mathcal{C}_W$ is invertible and the solution to Problem \ref{RHP-1} is given by 
\begin{equation}
\label{M}
M(x,z)=I+\dfrac{1}{2\pi i}\int_\bbR \dfrac{\mu(x,\lambda)\left( W_{x+}+W_{x-} \right) }{\lambda-z}d\lambda.
\end{equation}
\end{proof}
From \eqref{M} we can reconstruct the potential $u, v$ by taking the limit \cite[(2.1.42)]{K14}
\begin{equation}
\label{reconstruction}
[u(x), v(x)]=-2i\lim_{z\to \infty}\left[ zM_-(x,z)_{21},  zM_-(x,z)_{31}  \right]
\end{equation}
Note that this reconstruction formula works for all three cases listed in Remark \ref{cases}.
The following two lemmas can be found in \cite{Zhou98}:
\begin{lemma}
\label{lemma-resol}
$\norm{(I-\mathcal{C}_W)^{-1}}{L^2}$ is bounded for any $x\in (c, \infty)$ and the norm of this resolvent operator only depends on $\norm{W_\pm}{H^{1/2+\eps}}$.
\end{lemma}
A detailed proof of this lemma can be find in  the proof Proposition 4.2 of \cite{JLPS18}. We notice that the embedding $H^{1}\hookrightarrow H^{1/2+\eps} $  is compact.  A combination of this result with the second resolvent identity will show that $\norm{(I-\mathcal{C}_W)^{-1}}{L^2}$ is Lipschitz continuous on any bounded subset of $H^{1}$.
\begin{lemma}
\label{lemma-decay}
For $x\geq 0$ we have 
\begin{equation}
\label{cauchy-bd}
 \norm{C^\pm _{W_{x\mp}} }{L^2} \leq \dfrac{\norm{\rho_1}{H^1}+\norm{\rho_2}{H^1}}{(1+x^2)^{1/2}}.
 \end{equation}
\end{lemma}
The proof is standard Fourier theory. See Lemma 2.3 of \cite{Zhou98}.
\begin{proposition}
\label{prop-decay}
If $\rho_1, \rho_2\in H^{1}(\bbR)$,  then $u(x), v(x)\in H^{0,1}(\bbR)$.
\end{proposition}
\begin{proof}
We first assume $x\geq0$. Combining \eqref{reconstruction} with \eqref{M}, we will study the 
 following integral:
 \begin{align}
\label{int-decay} 
\int_\bbR \mu(x,\lambda)\left( W_{x+}+W_{x-} \right)d\lambda&=\int_\bbR (I-\mathcal{C}_W)^{-1} I \left( W_{x+}+W_{x-} \right) d\lambda\\
\nonumber
                                                                                                 &=\int_1+\int_2+\int_3
 \end{align}
 Where
 \begin{align*}
 \int_1 &=\int_\bbR \left( W_{x-}+W_{x+} \right)\\
 \int_2 &= \int_\bbR \left( \mathcal{C}_W I\right)\left( W_{x-}+W_{x+} \right)\\
 \int_3 &=\int_\bbR \left( \calC_W (\mu-I)  \right)\left( W_{x-}+W_{x+} \right)
 \end{align*}
 Clearly 
 $$\left\vert \int_1\right\vert \lesssim \dfrac{1}{1+x^2}$$ 
 can be obtained using standard Fourier theory. It is also clear that $\int_2$ is diagonal thus making no contribution to the (2-1) and (3-1) entries. For $\int_3$, we first note that by the previous two lemma,
 \begin{align}
 \label{mu-L2}
 \norm{\mu-I}{L^2}=\norm{(I-\calC_W)^{-1}\calC_W I }{L^2}\lesssim (1+x^2)^{-1/2}.
 \end{align}

Using the identity $C^+-C^-=I$, the triangularity of $W_\pm$ and the fact that $(C^\pm f)(C^\pm g)(z)$ has analytic extension to $\bbC^\pm$, an application of Cauchy's theorem gives 
\begin{align*}
\int_3 &=\int_\bbR \left( C^+(\mu-I)W_{x-}   \right)W_{x+}+\left( C^-(\mu-I)W_{x+}   \right)W_{x-}\\
         &=\int_\bbR \left( C^+(\mu-I)W_{x-}   \right)C^-(W_{x+})+\left( C^-(\mu-I)W_{x+}   \right)C^+(W_{x-})
\end{align*}
Lemma \ref{lemma-decay} and  the result of \eqref{mu-L2} together with Schwarz inequality and the fact that $H^{1}(\bbR)$ is an algebra imply
  $$\left\vert \int_3\right\vert \lesssim  \dfrac{1}{1+x^2}.$$ 
  Now we have shown that $xu(x), xv(x)\in L^2(\bbR)$ for $x\geq 0$. For the estimates when $x\leq 0$, we solve the Riemann-Hilbert problem with jump relation on $\bbR$ given by \eqref{BC-jump-til}. It is easy to check that $\tilde{V}$ admits the following triangular factorization 
  $$\tilde{V}=\threemat{1}{0}{0} {\tilde{\rho}_1}{1}{0} {\tilde{\rho_2}}{0}{1}\threemat{1}{\eps\tilde{\rho_1}^*}{\eps \tilde{\rho_2}^*} {0}{1}{0} {0}{0}{1}$$
that is opposite compared to that of $V$, so we can obtain the parallel decay results for $x\leq 0$. Also notice that for the auxiliary scattering matrix $A(z)$ given by \eqref{auxiliary}, we have $\lim_{z\to \infty}A(Z)=I$. Thus by \eqref{m-tilde} and \eqref{reconstruction}, we conclude that the two Riemann-Hilbert problems share the same reconstructed potential.
  \end{proof}
  \begin{proposition}
  \cite[Section 8]{Zhou89} The following relation holds:
  \begin{equation}
  \label{mu-diff}
  \dfrac{d}{dx}\mu(x,\lambda)=i\lambda  \ad \sigma \mu(x,\lambda)+ U(x) \mu(x,\lambda)
  \end{equation}
  where $U$ is formulated by the reconstructed potential \eqref{reconstruction}.
  \end{proposition}
  \begin{proposition}
  \label{smooth}
If $\rho_1, \rho_2\in H^{1,1}(\bbR)$,  then $u(x), v(x)\in H^1(\bbR)$.
\end{proposition}
\begin{proof}
Using the relation \eqref{mu-diff} and the fact that $\ad\sigma$ is a derivation, we have that
$$\int_\bbR \dfrac{d}{dx}\left(\mu e^{ix\lambda\ad\sigma}\left( W_{+}+W_-  \right)  \right) d\lambda= \int_\bbR \left( i\lambda\ad\sigma +U  \right)\left(\mu e^{ix\lambda\ad\sigma}\left( W_{+}+W_-  \right)  \right) d\lambda. $$
Using the fact that $\rho_1, \rho_2\in H^{1,1}$ we can follow the same argument in the proof of Proposition \ref{prop-decay} to see that 
$$\int_\bbR  i\lambda\ad\sigma \left(\mu e^{ix\lambda\ad\sigma}\left( W_{+}+W_-  \right)  \right) d\lambda \in L^2_x .$$
Moreover, Proposition \ref{prop-decay} and an application of Cauchy Schwarz inequality imply
$$\int_\bbR U \left( \mu e^{ix\lambda\ad\sigma}\left( W_{+}+W_-  \right)  \right) d\lambda \in L^2_x .$$
\end{proof}
\begin{remark}
\label{smooth^2}
If we further assume that $u_0, v_0\in H^2(\bbR)$, then by Remark \ref{lam^2}  we have that $\rho_1, \rho_2\in H^{0,2}(\bbR)$. So we are allowed to take the second order derivative in $x$ and show 
$$\int_\bbR \left( i\lambda\ad\sigma \right)^2\left(\mu e^{ix\lambda\ad\sigma}\left( W_{+}+W_-  \right)  \right) d\lambda \in L^2_x $$
which will lead to $u, v \in H^2(\bbR)$.
\end{remark}
\section{Case II: Generic Condition }
Throughout this section we set $\eps=1$ and assume that $s_{11}(z)$ has finitely many simple zeros $\lbrace z_i \rbrace_{i=1}^n$ in $\bbC^+$.  From \eqref{det-s11} we have the following linear dependence relation :
\begin{equation}
\label{linear}
c_{1,i}\psi^-_1(x, z_i)+c_{2,i}\psi^-_2(x, z_i)+c_{3,i}\psi^-_3(x, z_i)=0
\end{equation}
which gives
\begin{equation}
\label{c-i}
m^-_{1}(x, z_i)=m^{++}(x, z_i) e^{2ix z_i}\mathbf{c}_i
\end{equation}
where 
$$\mathbf{c}_i=-\dfrac{ (c_{2,i}, c_{3,i})^T}{c_{i, 1}}$$
Similarly in $\bbC^-$ \cite[(2.1.36), (2.1.37)]{K14}, 
\begin{equation}
\label{c-i}
\brm^{--}(x, z_i^*)=m^{+}_1(x, z_i) e^{-2ix z_i^*}(-\mathbf{c}_i^*)^T.
\end{equation}
We now define the following residue condition for the Beals-Coifman solution $M^\pm$:

\begin{subequations}
\begin{equation}
\label{res-1}
\Res_{z=z_i} M^+=\lim_{z\to z_i}M^+(x, z)e^{2i z_i x\ad\sigma}\twomat{0}{0}{\mathbf{C}_i}{0}
\end{equation}
\begin{equation}
\label{res-2}
\Res_{z=z_i} M^-=\lim_{z\to z_i^*}M^-(x, z)e^{2i z_i x\ad\sigma}\twomat{0}{(-\mathbf{C}_i^*)^T}{0}{0}
\end{equation}
\end{subequations}
where $\mathbf{C}_i=\mathbf{c}_i/s_{11}'(z_i)$.

Since $s_{11}(\lambda)\neq 0$ for $\lambda\in \bbR$,  $\rho_1(\lambda), \rho_2(\lambda)$ is well-defined for all $\lambda\in\bbR$. Following the same proof in the previous problem, we have the following proposition:
\begin{proposition}
\label{D2}
If $u_0, v_0\in \mathcal{G}\subset H^{1,1}(\bbR)$,  then $ \lbrace (\rho_1(\lambda), \rho_2(\lambda) ),  \lbrace z_i \rbrace_{i=1}^n,   \lbrace\mathbf{C}_i \rbrace_{i=1}^n \rbrace\in H^{1,1}(\bbR) \otimes (\bbC^+)^n \otimes (\bbC \times \bbC)^n$.
\end{proposition}
We now construct the Riemann-Hilbert problem for the inverse scattering part of Case II.
\begin{problem}
\label{RHP-2}
For fixed $x\in\bbR$ and  $ \lbrace (\rho_1(\lambda), \rho_2(\lambda) ),  \lbrace z_i \rbrace_{i=1}^n,   \lbrace\mathbf{C}_i \rbrace_{i=1}^n \rbrace\in H^{1,1}(\bbR) \otimes (\bbC^+)^n \otimes (\bbC \times \bbC)^n$, find a matrix $M(x,z)$ satisfying the following conditions:
\begin{enumerate}
\item[(i)](Analyticity) $M(x, z)$ is analytic for $z$ for $z\in \mathbb{C}\setminus\Lambda $  where
$$\Gamma= \mathbb{R} \cup \lbrace  \gamma_1, ..., \gamma_n, \gamma_1^*, ..., \gamma_n^* \rbrace  $$
\item[(ii)] (Normalization) 
$M(x,z)=I+\mathcal{O}(z^{-1})$ as $z \rightarrow\infty$,

\item[(iii)] (Jump condition) For each $\lambda\in\Gamma $, $M(x, z)$ has continuous 
boundary values
$M_{\pm}(\lambda)$ as $z \rarr \Gamma$. 
Moreover, the jump relation 
$$M^+(x,\lambda)=M^-(x,\lambda)e^{i\lambda x\ad\sigma}V(\lambda)$$ 
holds, where for $\lambda\in\mathbb{R}$
$$
V(\lambda)=V(\lambda)=\threemat{1+ |\rho_1(\lambda)|^2+|\rho_2(\lambda)|^2}{ \rho^*_1(\lambda)}{\rho^*_2(\lambda)}{\rho_1(\lambda)}{1}{0}{\rho_2(\lambda)}{0}{1}
$$
 and for $\lam \in \gamma_i\cup\gamma_i^* $ 
\begin{equation}
\label{jump-res}
V(\lambda) = 	\begin{cases}
						\twomat{1}{0}{\dfrac{ \bfC_i }{\lambda-z_i}}{1}	&	\lambda \in \gamma_i, \\
						\\
						\twomat{1}{\dfrac{ (\mathbf{C_i}^*)^T  }{\lambda-z_i^*} }{0}{1}
							&	\lambda \in \gamma_i^*
					\end{cases}
\end{equation}

\end{enumerate}
\end{problem}
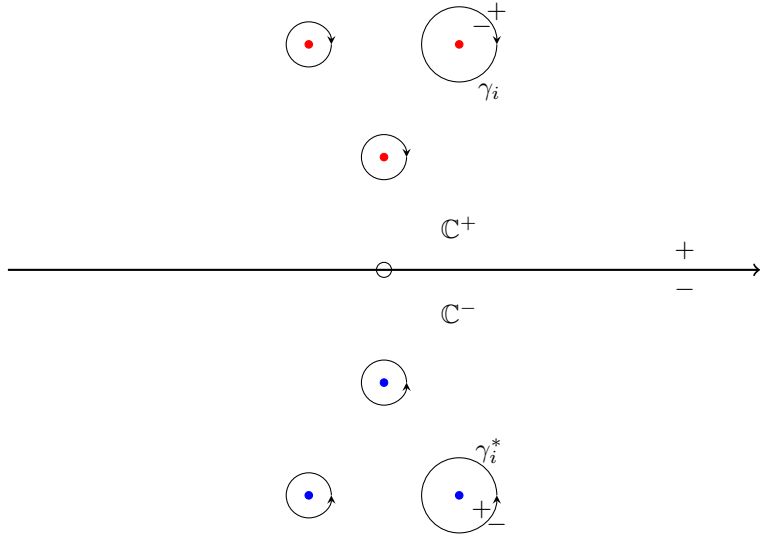
\begin{figure}[h]
\caption{The Contour $\Gamma$}
\begin{center}
\label{figure-1}
\begin{tikzpicture}
\draw [red, fill=red] (0,1.5) circle [radius=0.05];
\draw[->,>=stealth] (0.3,1.5) arc(360:0:0.3);
\draw [red, fill=red] (-1,3) circle [radius=0.05];
\draw[->,>=stealth] (-0.7,3) arc(360:0:0.3);

\draw [red, fill=red] (1,3) circle [radius=0.05];
\draw[->,>=stealth] (1.5,3) arc(360:0:0.5);
\node [below] at (1.4, 2.6) { $\gamma_i$};
\node[above] at(1.3,3.0){$-$};
\node[above] at(1.5,3.2){$+$};
\draw  (0,0) circle [radius=0.1];
\node [above] at (1,0.3) {$\mathbb{C^+}$};
\draw [->][thick] (-5,0) -- (5,0);
\node [below] at (1,-0.3) {$\mathbb{C^-}$};
\draw [blue, fill=blue] (0,-1.5) circle [radius=0.05];
\draw[->,>=stealth] (0.3,-1.5) arc(0:360:0.3);

\node [below] at (4, 0.5) { +};
\node [above] at (4, -0.5) {$-$};

\draw [blue, fill=blue] (1,-3) circle [radius=0.05];
\draw[->,>=stealth] (1.5,-3) arc(0:360:0.5);
\node [above] at (1.4, -2.7) { $\gamma_i^*$};
\node[below] at(1.3,-2.95){$+$};
\node[below] at(1.5,-3.15){$-$};

\draw [blue, fill=blue] (-1,-3) circle [radius=0.05];
\draw[->,>=stealth] (-0.7,-3) arc(0:360:0.3);

\end{tikzpicture}
\end{center}
\end{figure}
\begin{remark}
 In Problem \ref{RHP-2} for each pole of $M$,  $z_i \in \mathbb{C}^+$, let $\gamma_i$ be a circle centered at $z_i$ of sufficiently small radius to be lie in the open upper half-plane and to be disjoint from all other circles. By doing so we replace the residue condition of the Riemann-Hilbert problem with Schwarz invariant jump conditions across a set of complete contours. The purpose of this replacement is to make use of 
\begin{enumerate}
\item Vanishing lemma of homogeneous RHPs from \cite[Theorem 9.3]{Zhou89}.
\item The Plemelj formula over complete contours.
\end{enumerate}
\end{remark}
\begin{proposition}
The Riemann-Hilbert problem \ref{RHP-2} has a unique solution.
\end{proposition}
\begin{proof}
By standard Riemann-Hilbert theory, the existence and uniqueness of the solution to Problem \ref{RHP-2} is determined by the existence and uniqueness of the following singular integral equation:
\begin{align}
\label{SIE-2}
\mu(x,\lambda)&=I+\mathcal{C}_W \mu(x,\lambda)\\
\nonumber
                        &=I+C^+_\Gamma (\mu W_{x-})(x,\lambda)+C^-_\Gamma (\mu W_{x+})(x,\lambda)
\end{align}
In this case, for $\lambda\in \gamma_i \cup\gamma_i^*$, $V(\lambda)$ admits the following trivial factorization:
\begin{align*}
V(\lambda) &=(I-W_-)^{-1}(I+W_+)=\threemat{1}{0}{0}{0}{1}{0}{0}{0}{1}\twomat{1}{0}{\dfrac{ \bfC_i }{\lambda-z_i}}{1}, \quad \lambda\in \gamma_i \\
V(\lambda) &=(I-W_-)^{-1}(I+W_+)=\twomat{1}{\dfrac{ (\mathbf{C_i}^*)^T  }{\lambda-z_i^*} }{0}{1}\threemat{1}{0}{0}{0}{1}{0}{0}{0}{1}, \quad \lambda\in \gamma_i ^*
\end{align*}

It is shown in \cite[Proposition 4.1, 4.2]{Zhou89} that the operator $I-\mathcal{C}_W$ is Fredholm and has Fredholm index zero. It is also easy to check that $V(\lambda)+V(\lambda)^\dagger$ is a positive definite matrix  for $\lambda\in\bbR$ and
$$V(\lambda)=V(\lambda^*)^\dagger$$ 
for $\lambda\in \Gamma\setminus\bbR$. Thus $ker(I-\mathcal{C}_W)=0$ by \cite[Proposition 9.3]{Zhou89}. The operator $I-\mathcal{C}_W$ is invertible and the solution to Problem \ref{RHP-2} is given by 
\begin{equation}
\label{M-II}
M(x,z)=I+\dfrac{1}{2\pi i}\int_\Gamma \dfrac{\mu(x,\lambda)\left( W_{x+}+W_{x-} \right) }{\lambda-z}d\lambda.
\end{equation}
\end{proof}
An application of the Cauchy's theorem gives us
\begin{align}
\label{Cauchy}
\int_\Gamma \mu(x,\lambda)\left( W_{x+}+W_{x-} \right)d\lambda &=\int_\bbR \mu(x,\lambda)\left( W_{x+}+W_{x-} \right)d\lambda\\
 \nonumber  
                                                                                                        &\quad+\sum_i\int_{\gamma_i} \mu(x,\lambda)\left( W_{x+}+W_{x-} \right)d\lambda\\
\nonumber                                                                                       &\quad+\sum_i\int_{\gamma_i^*} \mu(x,\lambda)\left( W_{x+}+W_{x-} \right)d\lambda\\                         
\nonumber
&=\int_\bbR \mu(x,\lambda)\left( W_{x+}+W_{x-} \right)d\lambda\\
\nonumber
&\quad +\sum_i \left( (\mu_2(x, z_i), \mu_3(x, z_i)) e^{2i z_i x}\bfC_i , \,\mu_1(x, z_i^*)e^{-2i z_i^*}  (\mathbf{C_i}^*)^T \right)
\end{align}
Given the fact that Lemma \ref{lemma-resol} works for the contour $\Gamma$, to prove the same decay estimates given by Proposition \ref{prop-decay}, we only have to obtain the bound \eqref{cauchy-bd} on $\Gamma$. Indeed, for $\lambda\in \bbR$
\begin{align*}
C^+_{W_{x-}} I&=\lim_{\eps\to 0+}\left( \dfrac{1}{2\pi i}\int_{\bbR} \dfrac{e^{-2i x s}W_-(s)}{s-(\lambda+i\eps)}ds+\dfrac{1}{2\pi i}\int_{\gamma_i}  \dfrac{e^{-2i x s} (\mathbf{C_i}^*)^T }{ (s-(\lambda+i\eps) ) (s-z_i^*)}ds \right)\\
                       &=C^+_\bbR W_-(\lambda)+ \dfrac{e^{-2i x z_i^*} (\mathbf{C_i}^*)^T }{ z_i^*-\lambda}
\end{align*}
and for $\lambda\in \gamma_i^*$
\begin{align*}
C^+_{W_{x-}} I&=\lim_{\lambda'\to\lambda}\left( \dfrac{1}{2\pi i}\int_{\bbR} \dfrac{e^{-2i x s}W_-(s)}{s-\lambda'}ds+\dfrac{1}{2\pi i}\int_{\gamma_i}  \dfrac{e^{-2i x s} (\mathbf{C_i}^*)^T }{ (s-\lambda' ) (s-z_i^*)}ds \right)\\
                       &= \dfrac{1}{2\pi i}\int_{\bbR} \dfrac{e^{-2i x s}W_-(s)}{s-\lambda}ds+ \dfrac{e^{-2i x z_i^*} (\mathbf{C_i}^*)^T }{ z_i^*-\lambda}+\dfrac{e^{-2i x \lambda} (\mathbf{C_i}^*)^T }{ \lambda-z_i^*}
 \end{align*}
Now we conclude that the residue condition only adds  some exponentially decaying components to the operator $\calC_W$ for $x\geq 0$. So we are allowed to repeat the proofs of Proposition \ref{prop-decay} and Proposition \ref{smooth} to obtain the following:
\begin{proposition}
\label{smooth-1}
If $ \lbrace (\rho_1(\lambda), \rho_2(\lambda) ),  \lbrace z_i \rbrace_{i=1}^n,   \lbrace\mathbf{C}_i \rbrace_{i=1}^n \rbrace\in H^{1,1}(\bbR) \otimes (\bbC^+)^n \otimes (\bbC \times \bbC)^n$, then $u(x), v(x)\in H^{1,1}(\bbR)$.
\end{proposition}
\section{Arbitrary Spectral Singularities}
In this section we allow $s_{11}(z)$ to have (possibly infinitely many) zeros  of arbitrary order in $\bbC^+\cup \bbR$. This leads to two consequences:
\begin{enumerate}
\item The reflection coefficients $\rho_1, \rho_2$ are not defined on $\bbR$.
\item The Beals-Coifman solution $M(x,z)$ may not have continuous limit as $z$ approaches $\bbR$ from $\bbC^\pm$.
\end{enumerate}
In a series of papers, \cite{Zhou89-2,Zhou95,Zhou98},  Zhou developed new tools to construct  direct and inverse scattering maps to overcome the obstacles above. The key idea is to make use of the following two facts of $s_{11}$
\begin{enumerate}
\item For $\norm{U}{L^1}\ll 1$, $s_{11}$ is non-zero in $\bbC^+\cup \bbR$. This follows from the analytic Fredholm theory.
\item $\lim_{\lambda\to \infty} s_{11}(\lambda)=1$. This follows from \eqref{S-int} and the Riemann-Lebesgue lemma.
\end{enumerate}
Let $\Sigma_\infty$ be a circle centered at $0$ such that $s_{11}$ does not have any zero outside the circle. We arrive at the following picture:
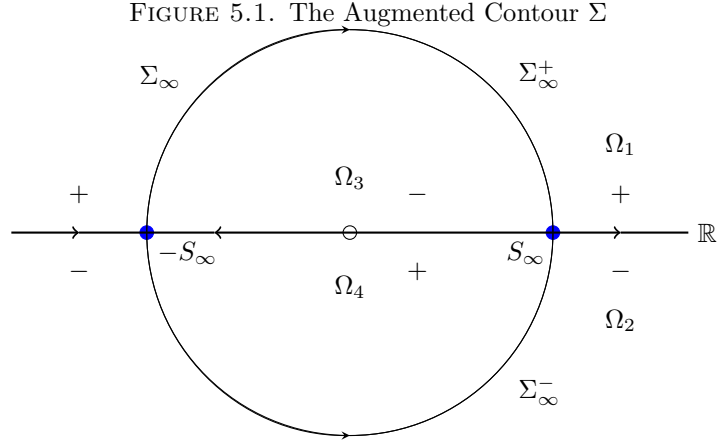
\begin{figure}[H]
\caption{The Augmented Contour $\Sigma$}
\label{figure-2}
\begin{tikzpicture}[scale=0.9]
\draw  (0,0) circle [radius=0.1];
\draw [blue, fill=blue] (3,0) circle [radius=0.1];
\draw  [blue, fill=blue] (-3,0) circle [radius=0.1];
\draw[->,>=stealth] (-3,0) arc(180:90:3);
\draw[->,>=stealth] (-3,0) arc(180:270:3);
\draw(0,3) arc(90:0:3);
\draw(0,-3) arc(270:360:3);
\draw (0,0) circle [radius=3];
\draw [->][thick] (-5,0) -- (-4,0);
\draw [-][thick] (-4,0) -- (-3,0);
\draw [->][thick] (3,0) -- (4,0);
\draw[-][thick] (4,0) -- (5,0);
\draw [<-][thick] (-2,0) -- (3,0);
\draw[-][thick] (-2,0) -- (-3,0);
\node [below] at (1,-0.3) {$+$};
\node [above] at (1,0.3) {$-$};
\node [below] at (4,-0.3) {$-$};
\node [above] at (4,0.3) {$+$};
\node [below] at (-4,-0.3) {$-$};
\node [above] at (-4,0.3) {$+$};
\node [below] at (0,-0.5) {$\Omega_4$};
\node [above] at (0,0.5) {$\Omega_3$};
\node [below] at (4,-1) {$\Omega_2$};
\node [above] at (4, 1) {$\Omega_1$};
\node [above] at (-2.8, 2) {$\Sigma_\infty$};
\node [right] at (5,0) {$\mathbb{R}$};
\node[below] at (2.6,0){$S_\infty$};
\node[below] at (-2.4,0){$-S_{\infty}$};
\node [above] at (2.8, 2) {$\Sigma_\infty^+$};
\node [below] at (2.8, -2) {$\Sigma_\infty^-$};
\end{tikzpicture}
\end{figure}
\subsection{ Construction of Scattering Data }
Let $x_0\in \bbR$ be such that the cut-off potential   $U_{x_0}=U\chi_{ (x_0, \infty )  }$ satisfies $\norm{U_{x_0}}{L^1}\ll 1$. The strategy is to construct a new analytic function which is also normalized at $x\to +\infty$ to replace $M(x,z)$ inside the circle $\Sigma_\infty$ and formulate a new Riemann-Hilbert problem along the augmented contour $\Sigma=\bbR\cup\Sigma_\infty$. We first note that we can still have $M(x,z)$ and $\rho_1, \rho_2$ outside the circle $\Sigma_\infty$ as given by \eqref{BC+}-\eqref{BC-} and \eqref{rho1,2}. We then construct a Beals-Coifman solution $M^{(0)}$ normalized as $x\to \infty$ associated to the  potential $U_{x_0}$. Note that this $M^{(0)}$ is constructed in parallel with the construction of $M(x,z)$.  Associated to $M^{(0)}$ are the scattering matrix $\mathbf{S}(\lambda)$ $(\mathbf{T}=\mathbf{S}^{-1})$ and reflection coefficients $r_1(\lambda), r_2(\lambda)$. More specifically,
\begin{equation}
\label{S-0}
\mathbf{S}(\lambda)=I+\int_{x_0}^\infty e^{-i\lambda y\ad\sigma}U(y)m^-(y,\lambda)dy
\end{equation}
and 
\begin{equation}
\label{r-1,2}
r_1(\lambda):=\dfrac{\bfs_{21}(\lambda)}{\bfs_{11}(\lambda)},\quad r_2(\lambda):=\dfrac{\bfs_{31}(\lambda)}{\bfs_{11}(\lambda)}
\end{equation}
We then define $M^{(1)}$ to be the solution to the following Volterra integral equation
\begin{equation}
\label{M1}
 M^{(1)}(x,z) = I + \int_{x_0}^x e^{i(x-y)z \ad\sigma}   U(y)M^{(1)} (y, z)dy
\end{equation}
and  
\begin{equation}
\label{M2}
 M^{(2)}(x,z) =M^{(1)}(x,z)e^{i(x-x_0)z\ad\sigma } M^{(0)}(x_0, z)
\end{equation}
From $M$ and $M_2$, we construct a new matrix-valued function $\bfM$ that is piecewise analytic in $\bbC$ with jump along the augmented contour in the following way.
Let $$\bfM(x,z)=M(x,z)$$
 on
$\Omega_1\cup \Omega_2$ and 
$$\bfM(x,z)=M^{(2 )}(x,z)$$
on $\Omega_3\cup \Omega_4$. 

We can explicitly compute the jump matrices across the various parts of the contour $\Sigma$. Along $(-\infty, -S_\infty )\cup (S_\infty, \infty)$, the jump matrix is given by:
 \begin{equation}
 \label{V-R-1}
V(\lambda)=\threemat{1+ |\rho_1(\lambda)|^2+|\rho_2(\lambda)|^2}{ \rho^*_1(\lambda)}{ \rho^*_2(\lambda)}{\rho_1(\lambda)}{1}{0}{\rho_2(\lambda)}{0}{1}.
\end{equation}
Along $(-S_\infty, S_\infty)$, the jump matrix is given by:
\begin{equation}
 \label{V-R-2}
V(\lambda)=\threemat{1}{- r^*_1(\lambda)}{- r^*_2(\lambda)}{-r_1(\lambda)}{1+ |r_1(\lambda)|^2}{r_1(\lambda) r_2^*(\lambda)}{-r_2(\lambda)}{r_1^*(\lambda) r_2(\lambda)}{1+ |r_2(\lambda)|^2}
\end{equation}
Along the circle $\Sigma_\infty$, we define
$$V(\lambda)=e^{-ix\lambda\ad\sigma} \left( M^{(2)}(x,\lambda)^{-1}M(x,\lambda) \right).$$
If we set $x=x_0$, we get $V(\lambda)$ explicitly in terms of Jost functions:
 \begin{enumerate}
 \item[1.] Across the arc in $\bbC^+$:
 \begin{align}
\label{v-circle+}
e^{ix_0\lambda\ad\sigma} V(\lambda) &= M^{(2)}(x_0,\lambda)^{-1}M(x_0,\lambda) \\
\nonumber
                                                         &= M^{(0)}(x_0,\lambda)^{-1}M(x_0,\lambda) \\
\nonumber
                                                      &=\threemat{1}{0}{0} {v_{21}^+(\lambda)} {1}{0}{v_{31}^+(\lambda)}{0}{1}
\end{align}
where 
\begin{subequations}
\begin{equation}
\label{v21+}
v_{21}^+=\dfrac{  m_{21}^-(x_0)  m_{33}^+(x_0)-m_{23}^+(x_0)m_{31}^-(x_0)        }{\bfs_{11}s_{11}}
\end{equation}
\begin{equation}
\label{v31+}
 v_{31}^+=\dfrac{  m_{31}^-(x_0)  m_{22}^+(x_0)-m_{32}^+(x_0)m_{21}^-(x_0)        }{\bfs_{11}s_{11}}
 \end{equation}
\end{subequations}
\item[2.] Across the arc in $\bbC^-$:
\begin{align}
\label{v-circle-}
e^{ix_0\lambda\ad\sigma} V(\lambda) &= M(x_0,\lambda)^{-1}M^{(2)}(x_0,\lambda) \\
\nonumber
                                                          &= M(x_0,\lambda)^{-1}M^{(0)}(x_0,\lambda) \\
\nonumber
                                                      &=\threemat{1}{v_{12}^-(\lambda)}{v_{13}^-(\lambda)} {0 } {1}{0}{0}{0}{1}
\end{align}
\begin{subequations}
\begin{equation}
\label{v12-}
v_{12}^-=\dfrac{  m_{21}^-(x_0)^*  m_{33}^+(x_0)^*-m_{23}^+(x_0)^*m_{31}^-(x_0)^*        }{\mathbf{t}_{11}t_{11}}
\end{equation}
\begin{equation}
\label{v13-}
 v_{13}^-=\dfrac{  m_{31}^-(x_0)^*  m_{22}^+(x_0)^*-m_{32}^+(x_0)m_{21}^-(x_0)^*        }{\mathbf{t}_{11}t_{11}}
\end{equation}
\end{subequations}
 \end{enumerate}
 
\begin{remark}
In the calculation of \eqref{v-circle+}, since we construct $M^{(0)}$ in parallel with the construction of $M$, we can write
$$M^{(0)}(x_0, \lambda)=\threemat{1/\bfs_{11}}{ m^+_{12}(x_0, \lambda)}{m^+_{13}(x_0, \lambda)} {0}{m^+_{22}(x_0, \lambda)}{m^+_{23}(x_0, \lambda)} {0}{m^+_{32}(x_0, \lambda)}{m^+_{33}(x_0, \lambda)}.$$
To arrive at the expression above, we made use of the following two facts:
\begin{enumerate}
\item 
\begin{align*}
m^-_0(x_0,\lambda)&=I+\int_{-\infty}^{x_0} e^{i\lambda (x_0-y)\ad\sigma}U_{x_0}(y)m^-_0(y,\lambda)dy\\
                            &=I
\end{align*}
\item
\begin{align*}
m^+(x_0,\lambda) &=I+\int_{+\infty}^{x_0} e^{i\lambda (x_0-y)\ad\sigma}U(y)m^+(y,\lambda)dy\\
                              &=I+\int_{+\infty}^{x_0} e^{i\lambda (x_0-y)\ad\sigma}U_{x_0}(y)m^+(y,\lambda)dy\\
                              &=m^+_0(x_0,\lambda)
\end{align*}
\end{enumerate}
The first fact is trivial while the second fact follows from the uniqueness theory of ODE. The calculation of \eqref{v-circle-} is similar.
\end{remark} 
 \begin{remark}
 The explicit expressions of \eqref{v-circle+} and \eqref{v-circle-} enable us to precisely check if the jump matrix $V(\lambda)$ satisfies the matching condition (see Definition \ref{Hk} below) at the self intersection points of the contour, namely $(\pm S_\infty, 0)$. This condition is necessary for showing that $V(\lambda)$ factorizes into a pair of decomposing algebra $H^{1}(\Sigma_\pm)$ along the contour $\Sigma$.  An important fact is that the Cauchy projections $C^\pm$ are bounded on $H^{1}(\Sigma_\pm)$ (see \cite[Proposition 2.1]{Zhou89}) and this will enable us to show various mapping properties of the inverse scattering map.
 \end{remark}
First of all, it is easy to verify the following proposition:
  \begin{proposition}
 \label{schwarz invariance}
The jump matrix $V$  on $\Sigma$ defined in \eqref{V-R-1}-\eqref{v-circle+} and \eqref{v-circle-} satisfies:
\begin{enumerate}
\item[(i)] $V(\lambda)+V(\lambda)^\dagger$ is positive definite for $\lambda\in\bbR$.
\smallskip
\item[(ii)] $V({\lambda^*})=V(\lambda)^\dagger$ for $\lambda\in \Sigma \setminus\bbR$.
\end{enumerate}
\end{proposition}

To give a  full characterization of  the scattering matrix we need some definitions. 
Following \cite{Zhou89} and  \cite[Chapter 2]{TO16} ,  we first recall the definition of Sobolev spaces $H^k$ along a contour $\partial \Omega$.
\begin{definition}
\label{Hk}
 Let $\Omega$ be an open connected region with piecewise smooth boundary $\partial \Omega$. Denote by $S_\Omega$ the set of non-smooth points on $\partial \Omega$. The space $H^k(\partial \Omega)$ $k\geq 0$ consists of functions on $\partial\Omega$ which satisfies:
 \begin{enumerate}
 \item[1.] The distributional derivative $f^{ ( j )  }\in L^2$ for $0\leq j\leq k$ on each curve segment of $\partial\Omega\setminus S_\Omega$.
 \item[2.] At each point $z'\in S_\Omega$, $f^{(j)}$, $0\leq j\leq k-1$ matches from two sides.
 \end{enumerate}
\end{definition} 
\begin{remark}
In our case, as is shown by Figure \ref{figure-2}, the contour $\Sigma$ consists of several parts. $(-S_\infty, 0)$ and $(S_\infty, 0)$ are the non-smooth points for $\partial\Omega_1$,  $\partial\Omega_2$,  $\partial\Omega_3$ and  $\partial\Omega_4 $.  They are also the self-intersection points of the contour $\Sigma$. A set of complete contours can be viewed simultaneously as the positive boundary of the positive region and the negative boundary of the negative region. Using $\Sigma$ in Figure \ref{figure-2} as an example, we have
\begin{align}
\label{Sigma_+}
\Sigma_+ &= \left[(-\infty, -S_\infty)\cup\Sigma^+_\infty\cup(S_\infty, \infty) \right]\cup \left[(-S_\infty, S_\infty)\cup\Sigma^-_\infty \right] \\
\nonumber
                &=\partial \Omega_1 \cup\partial \Omega_4
\end{align}
\begin{align}
\label{Sigma_-}
\Sigma_- &= \left[(-\infty, -S_\infty)\cup\Sigma^-_\infty\cup(S_\infty, \infty) \right]\cup \left[(-S_\infty, S_\infty)\cup\Sigma^+_\infty \right]\\
\nonumber
               &=\partial \Omega_2 \cup\partial \Omega_3
\end{align}
\end{remark}
 The scattering data is characterized in the following proposition:
\begin{proposition}
\label{scattering data}

The jump matrix $V$  on $\Sigma$ defined in \eqref{V-R-1}, \eqref{v-circle+} and \eqref{v-circle-} admits a triangular factorization $V(\lambda)=V_-^{-1}(\lambda)V_+(\lambda)$ where 
\begin{enumerate}
\item[1.]
$V_{\pm}(\lambda)-I\in H^{1,1}(\Sigma_\pm) $;
\item[2.] $V_+\restriction_{\partial\Omega_1}-I$ and $V_-\restriction_{\partial\Omega_3}-I$ are strictly lower triangular while $V_-\restriction_{\partial\Omega_2}-I$ and $V_+\restriction_{\partial\Omega_4}-I$ are strictly upper triangular.
\end{enumerate}
\end{proposition}
\begin{proof}
On $(-\infty, -S_\infty)\cup (S_\infty, \infty)$ the scattering matrix $V(\lambda)$ admits the following factorization:
\begin{align}
\label{V-1}
V(\lambda) &=V_{-}(\lambda)^{-1}V_{+}(\lambda)\\
\nonumber
                    &=\threemat{1}{\rho^*_1(\lambda) }{ \rho^*_2(\lambda) } {0}{1}{0} {0}{0}{1}\threemat{1}{0 }{0} {\rho_1(\lambda)}{1}{0} {\rho_2(\lambda)}{0}{1}
\end{align}
On $(-S_\infty, S_\infty)$ the scattering matrix $V(\lambda)$ admits the following factorization:
\begin{align}
\label{V-2}
V(\lambda)&=V_{-}(\lambda)^{-1}V_{+}(\lambda)\\
\nonumber
                  &=\threemat{1}{0}{0}{-r_1(\lambda)}{1}{0}{-r_2(\lambda)}{0}{1}  \threemat{1}{- r^*_1(\lambda)}{- r^*_2(\lambda)}{0}{1}{0}{0}{0}{1}.
\end{align}
Using the relation given by \eqref{S-m}, at the self-intersection point $\pm S_\infty$, from \eqref{v21+} and \eqref{v31+} we deduce
\begin{align*}
e^{-2i ( \pm S_\infty ) x_0}v_{21} &=e^{-2i (\pm S_\infty )x_0}\dfrac{m_{33}^+(x_0)\left( m^+_{21}(x_0)  s_{11}+ m^+_{22}(x_0)  s_{21}+m^+_{23}(x_0)  s_{31} \right)   }{\bfs_{11}(\pm S_\infty)s_{11}(\pm S_\infty)}\\
            &\quad -e^{-2i (\pm S_\infty )x_0}\dfrac{m_{33}^+(x_0)\left( m^+_{21}(x_0)  s_{11}+ m^+_{22}(x_0)  s_{21}+m^+_{23}(x_0)  s_{31} \right)   }{\bfs_{11}(\pm S_\infty)s_{11}(\pm S_\infty)}\\
           &=e^{-2i (\pm S_\infty ) x_0} \dfrac{  \left(m_{33}^+(x_0)m_{21}^+(x_0) -m_{23}^+(x_0)m_{31}^+(x_0)  \right) s_{11}}{\bfs_{11}(\pm S_\infty)s_{11}(\pm S_\infty)}\\
           &\quad +e^{-2i (\pm S_\infty ) x_0}  \dfrac{  \left(m_{33}^+(x_0)m_{22}^+(x_0) -m_{23}^+(x_0)m_{32}^+(x_0)  \right) s_{21}}{\bfs_{11}(\pm S_\infty)s_{11}(\pm S_\infty)}\\
           &=-r_1(\pm S_\infty)+\rho_1(\pm S_\infty)
\end{align*}
and
\begin{align*}
e^{-2i (\pm S_\infty ) x_0}v_{31} &=e^{-2i (\pm S_\infty ) x_0}\dfrac{m_{22}^+(x_0)\left( m^+_{31}(x_0)  s_{11}+ m^+_{32}(x_0)  s_{21}+m^+_{33}(x_0)  s_{31} \right)   }{\bfs_{11}(\pm S_\infty)s_{11}(\pm S_\infty)}\\
            &\quad -e^{-2i (\pm S_\infty ) x_0}\dfrac{m_{32}^+(x_0)\left( m^+_{21}(x_0)  s_{11}+ m^+_{22}(x_0)  s_{21}+m^+_{23}(x_0)  s_{31} \right)   }{\bfs_{11}(\pm S_\infty)s_{11}(\pm S_\infty)}\\
           &=e^{-2i (\pm S_\infty )x_0} \dfrac{  \left(m_{31}^+(x_0)m_{22}^+(x_0) -m_{21}^+(x_0)m_{32}^+(x_0)  \right) s_{11}}{\bfs_{11}(\pm S_\infty)s_{11}(\pm S_\infty)}\\
           &\quad +e^{-2i (\pm S_\infty ) x_0}  \dfrac{  \left(m_{33}^+(x_0)m_{22}^+(x_0) -m_{23}^+(x_0)m_{32}^+(x_0)  \right) s_{21}}{\bfs_{11}(\pm S_\infty)s_{11}(\pm S_\infty)}\\
           &=-r_2(\pm S_\infty)+\rho_2(\pm S_\infty)
\end{align*}
We conclude that at the point $(\pm S_\infty, 0)$ the scattering matrix $V\restriction_{\Sigma_\infty^+}$ admits the following factorization:
\begin{align}
\label{V+S-inf}
V\restriction_{\Sigma_\infty^+}(\pm S_\infty)&=V_{-}(\pm S_\infty)^{-1}V_{+}(\pm S_\infty)\\
\nonumber
                 &=\threemat{1}{0}{0}{-r_1(\pm S_\infty)}{1}{0}{-r_2(\pm S_\infty)}{0}{1}  \threemat{1}{0}{0}{\rho_1(\pm S_\infty)}{1}{0}{\rho_2(\pm S_\infty)}{0}{1}.
\end{align}
Similarly at the point $(\pm S_\infty, 0)$ the scattering matrix $V\restriction_{\Sigma_\infty^-}$ admits the following factorization:
\begin{align}
\label{V-S-inf}
V\restriction_{\Sigma_\infty^-}(\pm S_\infty)&=V_{-}(\pm S_\infty)^{-1}V_{+}(\pm S_\infty) \\
\nonumber
                     &=\threemat{1}{-r_1^*(\pm S_\infty)}{-r_2^*(\pm S_\infty)}{0}{1}{0}{0}{0}{1}  \threemat{1}{\rho_1^*(\pm S_\infty)}{\rho_2^*(\pm S_\infty)}{0}{1}{0}{0}{0}{1}.
\end{align}
Now for $i=1,2$,  let $L_i(\lambda)$ be the linear polynomial such that 
$$L_i(\pm S_\infty)= e^{2i (\pm S_\infty) x_0}\rho_i(\pm S_\infty)$$
then we can rewrite \eqref{v-circle+} as
\begin{align}
\label{V3}
V(\lambda)&=V_-(\lambda)^{-1}V_+(\lambda)\\
\nonumber
                   &=e^{-i\lambda x_0 \ad\sigma}\left[ \threemat{1}{0}{0}{  v_{21}^{+}(\lambda)-L_1(\lambda) }{1}{0}{ v_{31}^{+}(\lambda)-L_2(\lambda) }{0}{1}\threemat{1}{0}{0}{L_1(\lambda)}{1}{0}{L_2(\lambda)}{0}{1} \right]
 \end{align}
 Similarly, we can rewrite \eqref{v-circle-} as
 \begin{align}
 \label{V4}
V(\lambda)&=V_-(\lambda)^{-1}V_+(\lambda)\\
\nonumber
                   &=e^{-i\lambda x_0 \ad\sigma} \left[ \threemat{1}{ v_{12}^{-}(\lambda) -L_1^*(\lambda)}{ v_{13}^{-}(\lambda)-L_2^*(\lambda)  } {0}{1}{0} {0}{0}{1}\threemat{1}{L_1^*(\lambda)}{L_2^*(\lambda) }{0}{1}{0}{0}{0}{1} \right].
 \end{align}
Now by comparing \eqref{V-1}-\eqref{V4} it is easy to check that $V_+(\lambda)\restriction_{\Omega_1}$ ($V_+(\lambda)\restriction_{\Omega_4}$ ) is a lower (upper) triangular matrix that satisfies the matching condition given by (ii) of Definition \ref{Hk}. Similarly, $V_-(\lambda)\restriction_{\Omega_3}$ ($V_-(\lambda)\restriction_{\Omega_2}$ ) is a lower (upper) triangular matrix that also satisfies the matching condition given by (ii) of Definition \ref{Hk}. The $L^2$ integrability of $V(\lambda)$ for $\lambda \in (-\infty, -S_\infty)\cup(S_\infty, +\infty)$ follows from the same proof as Proposition \ref{D1}.
\end{proof}
We have chosen $x_0\in \bbR$ to be such that the cut-off potential   $U_{x_0}=U\chi_{ (x_0, \infty )  }$ satisfies the condition $\norm{U_{x_0}}{L^1}\ll 1$. Without loss of generality we further make the assumption that the cut-off potential   $\tilde{U}_{x_0}=U\chi_{ (- \infty, -x_0 )}$ also atisfies the condition$\norm{\tilde{U}_{x_0}}{L^1}\ll 1$.
Let $\tilde{\bfM }$ be constructed as $\bfM$ from potential $ \tilde{U}_{x_0}$ with normalization at $x\to-\infty$. Then we define the auxiliary matrix $s$:
\begin{equation}
\label{auxiliary}
A(z)=e^{-ix z \ad\sigma}\tilde{\bfM}^{-1}(x,z)\bfM(x,z)
\end{equation}
For instance it is easy to see that for $z\in \Omega_1\cup\Omega_2$
$A(z) $ is given by \eqref{auxiliary}.
The jump matrix $\tilde{V}$ for $\tilde{ \bfM }$ is obtained from $V$ by conjugation
$$\tilde{V}=A_-^{-1}V A_+.$$
In analogy to Proposition \ref{scattering data} we have the following proposition:
\begin{proposition}
\label{scattering data tilde}
The matrix $\tilde{V}=A_-^{-1}V A_+$ admits a triangular factorization $\tildV(\lambda)=\tildV_-^{-1}(\lambda) \tildV_+(\lambda)$ where 
\begin{enumerate}
\item[1.]
$\tildV_{\pm}(\lambda)-I\in H^{1,1}(\Sigma_\pm) $;
\item[2.] $\tildV_+\restriction_{\partial\Omega_1}-I$ and $\tildV_-\restriction_{\partial\Omega_3}-I$ are strictly upper triangular while $\tildV_-\restriction_{\partial\Omega_2}-I$ and $\tildV_+\restriction_{\partial\Omega_4}-I$ are strictly lower triangular.
\end{enumerate}
\end{proposition} 
\subsection{Inverse Scattering}
We study the following Riemann-Hilbert problem:
\begin{problem}
\label{RHP-3}
Given scattering data characterized by proposition \ref{scattering data},  for fix $x \in \bbR$, find $\bfM(x,\dotarg)$ with the following properties:
\begin{enumerate}
\item[(i)](Analyticity) $\bfM(x, z)$ is a  matrix-valued  analytic function of $z$ for $z\in \mathbb{C}\setminus\Gamma $  where
$$\Sigma= \mathbb{R} \cup \Sigma_\infty  $$
\item[(ii)] (Normalization) 
${\bfM}(x,z)=I+\mathcal{O}(z^{-1})$ as $z \rightarrow\infty$.
 
\item[(iii)] (Jump condition) For each $\lambda\in\Sigma $, $\bfM$ has continuous 
boundary values
$\bfM_{\pm}(\lambda)$ as $z \rarr \lambda$ from $\Omega_\pm$. 
Moreover, the jump relation 
$$\bfM_+(x,\lambda)=\bfM_-(x,\lambda) V_x(\lambda)$$ 
holds, where for $\lambda\in(-\infty, -S_\infty)\cup(S_\infty, \infty)$
$$
V_x(\lambda)=e^{i\lambda x \ad \sigma}\threemat{1+ |\rho_1(\lambda)|^2+|\rho_2(\lambda)|^2}{ \rho^*_1(\lambda)}{ \rho^*_2(\lambda)}{\rho_1(\lambda)}{1}{0}{\rho_2(\lambda)}{0}{1}
$$
 and for $\lam \in(-S_\infty, S_\infty) $ 
 $$ V_x(\lambda)= e^{i\lambda x \ad \sigma} \threemat{1}{- r^*_1(\lambda)}{- r^*_2(\lambda)}{-r_1(\lambda)}{1+ |r_1(\lambda)|^2}{r_1(\lambda) r_2^*(\lambda)}{-r_2(\lambda)}{r_1^*(\lambda) r_2(\lambda)}{1+ |r_2(\lambda)|^2}$$
 and for $\lambda\in\Sigma_\infty $
$$
V_x(\lambda) = 	\begin{cases}
						e^{i\lambda (x-x_0 )\ad \sigma}\threemat{1}{0}{0} {v_{21}^+} {1}{0}{v_{31}^+}{0}{1}	&	\lambda \in \Sigma_\infty^+, \\
						\\
						e^{i\lambda (x-x_0) \ad \sigma}\threemat{1}{v_{12}^-}{v_{13}^-} {0 } {1}{0}{0}{0}{1}
							&	\lambda \in  \Sigma_\infty^-
					\end{cases}
$$

\end{enumerate}
\end{problem}
\begin{proposition}
The Riemann-Hilbert problem \ref{RHP-3} has a unique solution.
\end{proposition}
\begin{proof}
By standard Riemann-Hilbert theory, the existence and uniqueness of the solution to Problem \ref{RHP-3} is determined by the existence and uniqueness of the following singular integral equation:
\begin{align}
\label{SIE-3}
\mu(x,\lambda)&=I+\mathcal{C}_W \mu(x,\lambda)\\
\nonumber
                        &=I+C^+_\Sigma (\mu W_{x-})(x,\lambda)+C^-_\Sigma (\mu W_{x+})(x,\lambda)
\end{align}
where we have given the explicit form of $W_\pm=\pm V_\pm \mp I$  in the proof of Proposition \ref{scattering data}. It is shown in \cite[Proposition 4.1, 4.2]{Zhou89} that the operator $I-\mathcal{C}_W$ is Fredholm and has Fredholm index zero. More importantly, we remark that the characterization of scattering matrix by Proposition \ref{scattering data} i.e. $V_\pm-I\in H^{1}(\Sigma_\pm)$ allows uniform approximation by rational functions in $L^\infty$ norm.  It also follows from Proposition \ref{schwarz invariance}  and \cite[Proposition 9.3]{Zhou89} that $ker(I-\mathcal{C}_W)=0$. Thus the operator $I-\mathcal{C}_W$ is invertible and the solution to Problem \ref{RHP-2} is given by 
\begin{equation}
\label{M-III}
M(x,z)=I+\dfrac{1}{2\pi i}\int_\Sigma \dfrac{\mu(x,\lambda)\left( W_{x+}+W_{x-} \right) }{\lambda-z}d\lambda.
\end{equation}
\end{proof}
\begin{proposition}
\label{prop-decay-1}
Given the jump matrix $V(\lambda)$ characterized by Proposition \ref{scattering data},   $u(x), v(x)\in H^{0,1}(\bbR)$.
\end{proposition}
We show the estimate for $x\geq 0$ first. Following  a reduction technique from \cite{Zhou98}, we  construct functions $\omega\in\bfA( {\mathbb{C}}\setminus \Sigma )$ such that for $k=1$
\begin{enumerate}
\item[1.] $\omega_\pm \in R(\partial \Omega_\pm )$ and  $\omega_\pm-I=O(z^{-2})$ as $z\rightarrow \infty$.
\item[2.] $\omega_\pm$ has the same triangularity as $V_\pm$, and
\item[3.] $\omega_\pm(z)=V_\pm(z)+o((z-a)^{k-1})$ for $a\in\bbR\cap \Sigma_\infty$.
\end{enumerate}
The construction of $\omega_\pm$  is given in \cite[Appendix I]{Zhou89-2}. For example,  consider the approximation of $V_+\restriction_{ \partial \Omega_1}$ . 
Since $(V_+ -I)\restriction _{\partial \Omega_1}$ is in $H^1$, we construct a rational function $\omega_+$ such that $(\omega_+-V_+)\restriction \partial \Omega_1$ vanishes at $\pm S_\infty$ to order  1. Explicitly 
$$\omega_+(\pm S_\infty)-I=\threemat{0}{0}{0}{\rho_1(\pm S_\infty)}{0}{0}{\rho_2(\pm S_\infty)}{0}{0}.$$
This is attained through the following steps: for $k=1,2$
\begin{enumerate}
\item[(i)]  Choose $z_k \notin \overline{\Omega}_1$ and denote $p_{k\pm}$ the Taylor polynomial of degree 0 of $(z-z_k)^n\rho_k(z)$ at $z=\pm S_\infty $. We choose $n\geq 3$.
\item[(ii)] By \cite[Lemma A1.2]{Zhou89-2}, there is  a polynomial $p_k(z)$ of degree at most $1$ such that
$$p_k(z) -  p_{k\pm}(z) = O(z \mp S_\infty).$$
\item[(iii)] Set $\omega_{k+}(z)= (z-z_k)^{-n}p_k(z)$. Clearly, $\omega_{k+}(z)-\rho_k(z)$   vanishes at $\pm S_\infty$ to  order  1. Since $n\geq 3$,  $\omega_+-I\in H^{1,1}(\partial\Omega_1)$ and $\omega$ is analytic in $\Omega_1$.
\end{enumerate}
 We have 
$$V=\omega_-^{-1}(V_-\omega_-^{-1})^{-1}(V_+\omega_+^{-1})\omega_+\equiv \omega_-^{-1}\mathcal{V}_-^{-1}\calV_+\omega_+\equiv \omega_-^{-1} \calV \omega_+.$$
The advantage of working with $\calV$ is that  
 $\calV_\pm$ vanishes at $\pm S_\infty$
 \begin{equation}
\label{calV}
\calV_\pm(\pm S_\infty)=I. 
\end{equation}
For   $x\geq0$, $\calV_x$ is the jump condition for the  Riemann-Hilbert problem
$$\mathcal{M}_+(x,\lambda)=\mathcal{M}_-(x,\lambda ) \calV_x(\lambda) \quad \lambda\in \Sigma$$
if and only if  $V_x$ is the jump condition for the Riemann-Hilbert problem \ref{RHP-3}. Here $\bfM=\mathcal{M} e^{i\lambda x\ad\sigma}\omega$ where $e^{i\lambda x\ad \sigma}\omega\in \bfA L^{\infty}(\mathbb{C}\setminus\Sigma)\cap \bfA L^2(\bbC \setminus\Sigma)$   is  guaranteed by  construction.Thus
$$\int_{\Sigma}\mu(x,\lambda) 
       e^{i\lambda x \ad\sigma} \left(W_+(\lambda) +W_- (\lambda) \right)d\lambda= \int_{\Sigma}\mu(x,\lambda) 
       e^{i\lambda x \ad\sigma} \left(\calW_+(\lambda) +\calW_- (\lambda) \right)d\lambda$$
which 
shows that   $\omega$ gives no contribution to the reconstruction of $q$ for $x\geq 0$.  Thus, we may  as well work with the new jump matrix $\calV=(I-\calW_-)^{-1}(I+\calW_+)$.

The next step consists of augmenting the contour as in  figure \ref{fig:new.mod.contour} below. The advantage of this new contour is that it reverses the orientation of the segment $ (S_\infty^-, S_\infty^+)$ and thus allows  to prove usual estimates of the Cauchy projections when the contour is restricted to $\bbR$. The added ellipse has no effect of the RHP since the jump matrices there are chosen to be the identity.
 \begin{figure}[H]
\caption{The newly  modified contour}

\vskip 0.5cm

\begin{tikzpicture}
\draw[thick,->]	(-6,0) 	-- 	(-4,0);
\draw[thick]		(-4,0)		--	(-2,0);
\draw[thick,->-]	(-2,0)			arc(180:0:2);
\draw[thick,->-]	(-2,0)	--	(2,0);
\draw[thick,->-]	(-2,0)			arc(180:360:2);
\draw[thick,->]	(2,0)	--	(4,0);
\draw[thick]		(4,0)	--	(6,0);
\draw[thick,dashed,->-]	(2,0) arc(0:180:2cm and 1cm);
\draw[thick,dashed,->-]		(2,0) arc(360:180:2 cm and 1 cm);
\draw[black,fill=black]		(-2,0)		circle(0.06cm);
\draw[black,fill=black]		(2,0)		circle(0.06cm);
\node[below] at (-2.75,0)	{$(-S_\infty,0)$};
\node[below] at (2.55,0)	{$(S_\infty,0)$};
\node[above]	 at (3.5,0)				{$+$};
\node[below]	 at (3.5,0)				{$-$};
\node[above]  at (1.25,0)				{$+$};
\node[below]	 at (1.25,0)				{$-$};
\end{tikzpicture}
\label{fig:new.mod.contour}
\end{figure}
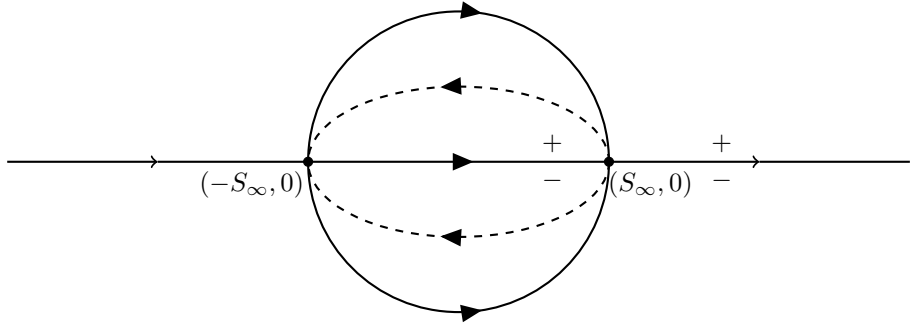

We redefine $\calV_\pm$ as follows:
\begin{enumerate}
\item[1.] $\calV_\pm=I$ on the added ellipse, 
\item[2.] $\calV_\pm(\lambda)$ is the lower/upper triangular factor in the lower/upper triangular factorization of $\calV$ ($\calV^{-1}$) on $\bbR$ for $|\lambda|>|S_\infty|$, ($|\lambda|<|S_\infty|$) and
\item[3.] for $\lambda\in \Sigma_\infty$, $\calV_\pm(\lambda)=I$ for $\imag \lambda\lessgtr 0$ and $\calV_\pm(\lambda)=\calV(\lambda)$ for $\imag \lambda\gtrless 0$.
\end{enumerate}
The newly defined $\calV_\pm$ satisfy all properties listed in Proposition \ref{scattering data}.
\begin{lemma}
\label{resolvent-3}
$\left\Vert (1-C_{\calW} )^{-1}\right\Vert_{L^2(\Sigma)\circlearrowleft }$ is uniformly bounded for any $x\in(c, \infty)$, $c\in\bbR$ where
$$C_{\calW  }\phi=C^+\phi(\calW_{x+} )+C^-\phi(\calW_{x-})  $$
\end{lemma}

\begin{definition}
\label{contour-Gamma}
We define the following contours which are subsets of the contour $\Gamma$:
\begin{equation}
\label{Gamma_pm}
\Sigma^\pm=\bbR \cup \left( \lbrace \imag \lambda \gtrless 0 \rbrace \cap \Sigma \right),
\end{equation}
Note that this is different from $\Sigma_\pm$ given by \eqref{Sigma_+}-\eqref{Sigma_-}.
\begin{equation}
\label{Gamma'}
\Sigma'  :=   \text{either}  \; \Sigma_\infty \cap \lbrace \imag \lambda \gtrless 0 \rbrace \, \, \text{or}  \, \,\bbR.
\end{equation}
\end{definition}
\begin{lemma}
For $x\geq 0$,
\begin{align}
\label{C+}
\left\Vert  C^+_{ \Sigma'\to\Sigma^+  }  (\calW_{x-})  \right\Vert_{L^2} &\leq \dfrac{c}{(1+x^2)^{1/2} }\norm{\calW_- }{H^1}\\
\label{C-}
\norm{ C^-_{ \Sigma'\to\Sigma^-  }  (\calW_{x+}) }{L^2} &\leq \dfrac{c}{ (1+x^2 )^{1/2}}\norm{\calW_+}{H^1}\\
\label{lower-c}
\norm{\calW_{x+}    }{L^2(\Sigma_\infty^+)} &\leq\dfrac{c}{(1+x^2)^{1/2}} \norm{\calW_+   }{H^1 } \\
\label{upper-c}
\norm{\calW_{x-}   }{L^2(\Sigma_\infty^-)} &\leq\dfrac{c}{(1+x^2)^{1/2}} \norm{\calW_-  }{H^1} \\
\label{CC}
 \norm{ (  \calC_{\calW  }  )^2 I}{L^2(\Sigma)} &\leq \dfrac{c}{(1+x^2)^{1/2}}\norm{\calW_+}{H^1}\norm{\calW_- }{H^1} 
\end{align}
\end{lemma}
\begin{proof}
See the proof of Lemma 2.9 in \cite{Zhou98}.
\end{proof}
\begin{proof}[Proof of Proposition \ref{prop-decay-1}]
The two previous lemmas provide the tools for estimating the decay of the potential $U$, recalling that $\mu$ appearing in \eqref{SIE-3} is equal to
$(I-\calC_{\calW})^{-1}I $.   
We will work with the following integral
\begin{equation}
\label{recon integral}
\int_\Sigma  \left( \left(  I -  \calC_{\calW}\right)^{-1} I \right) e^{i\lambda x\ad\sigma}(\calW_+ +\calW_+)d\lam=\int_1+\int_2+\int_3+\int_4
\end{equation}
where
\begin{equation}
\label{int-1}
\int_1=\int_\bbR \left( \calW_{x+} +\calW_{x-} \right)+\int_{\Sigma_\infty^+} \calW_{x+}+\int_{\Sigma_\infty^-} \calW_{x-}
\end{equation}
By  \cite[Lemma 2.9]{Zhou98} we conclude that the (2-1) and (3-1) entries of the integral above is in $H^{0,2}$ . The second integral
\begin{equation}
\label{int-2}
\int_2=\int_\Sigma \left( \calC_{\calW} \right) (\calW_{x+} +\calW_{x-})
\end{equation}
is diagonal thus makes no contribution to the reconstruction of $U$. For the third integral
\begin{align}
\label{int-3}
\int_3 &=\int_\Sigma \left(  (\calC_{\calW} )^2 I \right) (\calW_{x+} +\calW_{x-})\\
\nonumber
          &=\int_\Sigma \left( C^+\left(C^-\left(\calW_{x+}\right)  \right) \calW_{x-} \right)\calW_{x+}\\
          \nonumber
          &\quad+ \int_\Sigma \left( C^-\left(C^+\left( \calW_{x-} \right)  \right) \calW_{x+} \right)\calW_{x-}
\end{align}
we notice that the lower diagonal part is given by 
\begin{align*}
\int_{\Sigma_\infty^+} &\left( C^+_{\Sigma^- \to \Sigma}\left(C^-_{\Sigma \to \Sigma^-}\left( \calW_{x+} \right)  \right) \calW_{x-} \right)\calW_{x+} \\
&+\int_\bbR \left( C^+_{\Sigma^- \to \Sigma}\left(C^-_{\Sigma \to \Sigma^-}\left( \calW_{x+} \right)  \right) \calW_{x-} \right) C^-_\bbR\left( \calW_{x+} \right)
\end{align*}
and from \eqref{C+} - \eqref{lower-c}  we conclude that
$$\left\vert \left( \int_3 \right)_{lower} \right\vert \leq \dfrac{c}{1+x^2}.$$ 
Finally we set $$g=(1-\calC_{\calW } )^{-1} \left(  (  \calC_{\calW  }  )^2 I \right)   $$
and write
\begin{equation}
\left\vert  \int_4\right\vert = \left\vert \int_\Sigma \left[  \left( C^+g ( \calW_{x-} ) \right) \calW_{x+} +\left( C^- g \left(  \calW_{x+}\right)  \right)  \calW_{x-}  \right]  \right\vert
\end{equation}
Again we notice that the lower diagonal part is given by 
\begin{align*}
\int_{\Sigma_\infty^+} \left( C^+ g \left(   \calW_{x-} \right)  \right) \calW_{x+}
+\int_\bbR  \left( C^+_{\Sigma^-\to\bbR} ~ g \left(  \calW_{x-} \right)  \right) C_\bbR^- \left( \calW_{x+}\right) 
\end{align*}
and from \eqref{C+}-\eqref{lower-c} and \eqref{CC} we conclude that 
$$\left\vert \left( \int_4 \right)_{lower} \right\vert \leq \dfrac{c}{1+x^2}.$$
The estimate on $(-\infty, 0]$ can be obtained by considering the Riemann-Hilbert problem with jump matrix characterized in Proposition \ref{scattering data tilde}. This concludes the proof of Proposition \ref{prop-decay-1}.
\end{proof}
Following the same argument in the proof of Proposition \ref{smooth}, we can establish the following proposition:
\begin{proposition}
\label{prop-smooth-1}
Given the jump matrix $V(\lambda)$ characterized by Proposition \ref{scattering data},   $u(x), v(x)\in H^{1}(\bbR)$.
\end{proposition}

\section*{Acknowledgements}
The author would like to thank Wang Deng-shan for suggesting this problem and Chen Gong for useful discussions. The author would also like to thank Catherine Sulem for many valuable remarks and corrections.

\end{document}